\documentclass[11pt,twoside]{article}
\usepackage{amsfonts}
\usepackage{fancyhdr}
\usepackage{titlesec}
\usepackage{cite}
\usepackage{ifthen}
\usepackage{amssymb}
\usepackage{fancyhdr}
\usepackage{titlesec}
\usepackage[arrow,matrix]{xy}
\usepackage{pifont}
\usepackage{setspace}
\usepackage{indentfirst}
\usepackage{amsmath,amssymb,amscd,amsthm,mathrsfs}
\input amssym.def

\newboolean{first}
\setboolean{first}{true}
\renewcommand{\headrulewidth}{0pt}
\def\R{\mathbb{R}}
\newfont{\aaa}{cmb10 at 19pt}
\newfont{\bbb}{cmb10 at 11pt}
\newtheorem{lemma}{Lemma}[section]

\newtheorem{theorem}{Theorem}[section]
\newtheorem{definition}{Definition}[section]
\newtheorem{remark}{Remark}[section]
\newtheorem{corollary}{Corollary}[section]

\newcommand{\beq}{\begin{equation}}
\newcommand{\eeq}{\end{equation}}
\newcommand{\bey}{\begin{eqnarray}}
\newcommand{\eey}{\end{eqnarray}}
\newcommand{\beyy}{\begin{eqnarray*}}
\newcommand{\eeyy}{\end{eqnarray*}}

\numberwithin{equation}{section}
\makeatletter
\def\@oddfoot{}
\makeatother

\begin{document}
\thispagestyle{empty} \thispagestyle{fancy} {
\fancyhead[RO,LE]{\scriptsize \bf 
} \fancyfoot[CE,CO]{}}
\renewcommand{\headrulewidth}{0pt}

\begin{center}
{\bf \LARGE Two Generalizations of Stampacchia Lemma and Applications} 

\vspace{3mm}

{\small \textsc{HAN Yingxiao} \quad  \textsc{FANG Mi} \quad   \textsc {XIA Liuye} \quad \textsc{GAO Hongya}\footnote{Corresponding author, email: ghy@hbu.cn.}}

{\small College of Mathematics and Information Science, Hebei University, Baoding, 071002, China}
\end{center}

\begin{abstract} \noindent We present two generalizations of the classical Stampacchia Lemma which contain a non-decreasing
non-negative function $g$, and give applications. As a first application, we deal with variational integrals of the form
$$
{\cal J} (u;\Omega) = \int_{\Omega}\ f(x,Du{(x)})dx.
$$
We consider a minimizer $u: \Omega \subset \mathbb R^n \to \mathbb R $ among all functions with a fixed boundary value $u_{\ast }$
 on $\partial \Omega$. Under some nonstandard growth conditions of the integrand $f(x,\xi)$
we derive some regularity results; as a second application, we consider elliptic equations of the form
$$
\begin{cases}
-\mbox {div} \left( a(x, u(x)) D u(x) \right) = f(x), & x \in \Omega,  \\
 u(x) = 0, & x \in {\partial \Omega},
\end{cases}
$$
under the conditions
$$
\frac {\alpha }{(1+|s|) ^\theta \ln ^\theta (e+|s|)} \le a (x,s) \le \beta, \ \ \ 0<\alpha \le \beta <\infty, \ \theta \ge 0,
$$
we obtain some regularity properties of its weak solutions.


\end{abstract}

\thispagestyle{fancyplain} \fancyhead{}
\fancyhead[L]{\textit{}\\
}


\section{Introduction and preliminaries.}
\noindent In the present paper we shall extend the classical Stampacchia Lemma to generalized versions which contain
a non-decreasing non-negative function $g$, and give some applications.
We first pay our attention to the classical Stampacchia Lemma, which can be found, for example, in \cite{Stampacchia}.
\begin{lemma}\label{Stampacchia Lemma}
Let $c, \alpha, \beta$ be positive constants and $k_0\in \R$. Let
$\varphi: [k_0,+\infty)$ $\rightarrow [0, +\infty)$ be non-increasing
and such that
\begin{equation}\label{stampacchia_assumption-1}
\varphi (h) \le \frac {c}{(h-k)^\alpha} [\varphi (k)]^\beta
\end{equation}
\noindent
for every $h,k$ with $h>k\ge k_0$. It results that:

{\bf (i)} if $\beta > 1$ then
\begin{equation*}
 \label{statement_stampacchia_beta>1}
\varphi(k_0 + d) = 0,
\end{equation*}
where
\begin{equation*}
 \label{d=}
d^\alpha =   c [\varphi(k_0)]^{\beta - 1} 2^{\frac {\alpha \beta} {\beta
- 1}};
\end{equation*}

 {\bf(ii)} if $\beta = 1$ then for any $k\ge k_0$,
\begin{equation*} \label{statement_stampacchia_beta=1}
\varphi (k) \le \varphi(k_0) e^{1 - (  c e)^{-\frac{1}{\alpha}} (k - k_0)};
\end{equation*}

 {\bf(iii)} if $0< \beta < 1$ and $k_0 > 0$ then for any $k\ge k_0$,
\begin{equation*}
 \label{statement_stampacchia_beta<1}
\varphi (k) \le
2^{\frac {\alpha}{(1-\beta)^2 }}
\left\{   c ^\frac {1}{1-\beta} +(2 k_0)^{\frac \alpha
{1-\beta}} \varphi(k_0) \right\}
\left(\frac 1 k \right) ^{\frac \alpha {1-\beta}}.
\end{equation*}
\end{lemma}

Stampacchia Lemma is a fundamental tool in dealing with regularity issues of solutions of elliptic partial differential equations as
well as minima of variational integrals, and is used till now repeatedly by many mathematicians, see \cite{Boccardo,Boccardo-Croce,Kovalevsky,Kovalevskii,GHR,GZM,GLW}. In \cite{GDHR},
the authors gave two generalizations of Stampacchia lemma and their applications to quasilinear elliptic systems; in \cite{Kovalevskii},
Kovalevskii and Voitovich dealt with a similar condition, but includes an additional factor on the right hand side; in \cite{GHR}, Gao, Huang and Ren gave an
alternative proof which seems to be more elementary; in \cite{GZM}, Gao, Zhang and
Ma generalized the Stampacchia lemma with an additional factor $h^{\theta \alpha } $ on the right hand side and provided applications to degenerate elliptic equations.
For some remarks on Stampacchia Lemma we refer to \cite{GLW}. For some other generalizations and applications we refer to \cite{GZH}.

For convenience of the reader, we list below a generalization of Stampacchia Lemma due to Gao, Zhang and Ma, see \cite{GZM}, which can be
used in dealing with divergence type degenerate elliptic equations.

\begin{lemma}\label{Stampacchia Lemma, Gao Zhang and Ma generalized}
Let $c, \alpha, \beta, k_0$ be positive constants and $0\le \theta < 1$. Let
$\varphi: [k_0,+\infty)$ $\rightarrow [0, +\infty)$ be non-increasing
and such that
\begin{equation}\label{stampacchia_assumption-2}
\varphi (h) \le \frac {ch^{\theta \alpha } }{(h-k)^\alpha} [\varphi (k)]^\beta
\end{equation}
\noindent
for every $h,k$ with $h>k\ge k_0>0$. It results that:
	
{\bf (i)} if $\beta > 1$ then
\begin{equation*}
\label{statement_stampacchia_beta>1}
\varphi(2L) = 0,
\end{equation*}
where
\begin{equation*}\label{the value of L}
L= \max \left\{2k_0,  c^{\frac 1 {(1-\theta ) \alpha}}  [\varphi (k_0)] ^{\frac {\beta -1}{ (1-\theta) \alpha}}
2 ^{\frac 1 {(1-\theta)\beta} \left(\beta +\theta +\frac 1 {\beta -1}\right)} \right\} >0;
\end{equation*}
	
{\bf(ii)} if $\beta = 1$ then for any $k\ge k_0$,
\begin{equation*} \label{generalied statement_stampacchia_beta=1}
\varphi (k)\le \varphi(k_0) e ^{1-\left( \frac {k-k_0}{\tau} \right) ^{1-\theta }},
\end{equation*}
where
\begin{equation*}\label{tau11}
\tau =\max \left\{k_0,\left(ce2^{\theta \alpha} \right) ^{\frac 1 {(1-\theta)\alpha}}, \left( ce 2^{\frac {(2-\theta ) \theta \alpha }{1-\theta }} (1- \theta ) ^\alpha \right)^{\frac 1 {(1- \theta ) \alpha}} \right\};
\end{equation*}
	
{\bf(iii)} if $0< \beta < 1$ then for any $k\ge k_0$,
\begin{equation*} \label{statement_stampacchia_beta<1}
\varphi (k) \le 2^{\frac { (1-\theta)\alpha}{(1-\beta)^2 }} \left\{ (c_1 2^{\theta \alpha})^\frac {1}{1-\beta} +(2 k_0)
^{\frac {(1-\theta)\alpha} {1-\beta}} \varphi(k_0) \right\} \left(\frac 1 k \right) ^{\frac
{\alpha(1-\theta)} {1-\beta}},
\end{equation*}
where
\begin{equation*} \label{c1andc2}
c_1=\max \left\{ 4^{(1-\theta) \alpha} c 2^{\theta \alpha}, c_2^{1-\beta} \right\}, \ \ c_2= 2 ^{\frac {(1-\theta )\alpha}
{(1-\beta)^2 }} \left[(c2^{\theta \alpha})^{\frac 1 {1-\beta}} + (2k_0) ^{\frac {(1-\theta)\alpha}
{1-\beta}} \varphi (k_0)\right].
\end{equation*}
\end{lemma}

We note that, the difference between (\ref{stampacchia_assumption-1}) and (\ref{stampacchia_assumption-2}) lies in the fact that there is a factor $h^{\theta \alpha}$ in the numerator of the right hand side of (\ref{stampacchia_assumption-2}). This is convenient when dealing with degenerate elliptic equations, see \cite{GZM}.

We now ask two questions: How are the situations when (\ref{stampacchia_assumption-2}) is replaced by
$$
\varphi (h) \le \frac {c h ^{\theta \alpha}}{(h-k)^\alpha \ln ^\alpha (e+h-k)} [\varphi (k)]^\beta, \ \ \forall h>k\ge k_0,
$$
or
$$
\varphi (h) \le \frac {c h^{\theta \alpha }\ln ^{\theta \alpha }(e+h) }{(h-k)^\alpha} [\varphi (k)]^\beta,   \ \ \forall h>k\ge k_0 \ ?
$$
In the present paper we shall deal with such situations. More precisely, we shall consider some more general cases, that is,
(\ref{stampacchia_assumption-2}) is replaced by
\begin{equation}\label{the first condition}
\varphi (h) \le \frac {ch^{\theta \alpha }}{g^\alpha (h-k)} [\varphi (k)]^\beta, \ \ \forall h>k\ge k_0,
\end{equation}
or
\begin{equation}\label{the second condition}
\varphi (h) \le \frac {c g^{\theta \alpha }(h) }{(h-k)^\alpha} [\varphi (k)]^\beta,   \ \ \forall h>k\ge k_0,
\end{equation}
where the function $g:[0,+\infty)\to [0,+\infty)$ be a non-decreasing function satisfying suitable conditions.

We next introduce such a function $g$ and its properties in details: let $g:[0,+\infty)\to [0,+\infty)$ be a
function satisfying the following assumptions:


(${\cal G}_1$) $g$ be a $C^1$, convex and non-decreasing function, with $g(0)=0$ and $g(t)>0$ for all $t>0$;

(${\cal G}_2$) there exists a constant $\mu \ge 1$ such that for every $\lambda > 1 $ and every $t>0$,
\begin{equation}\label{condition for g-1}
g(\lambda t) \leq \lambda^{\mu } g(t);
\end{equation}


(${\cal G}_3$) $g'(0_+)>0$.

We notice that the above condition (\ref{condition for g-1}) means that, although the function $g$ is non-decreasing, it cannot grow too fast.
An example for a function $g$ satisfying (${\cal G}_1$), (${\cal G}_2$) and (${\cal G}_3$) is
\begin{equation}\label{condition for g-2}
g(t) =t \ln (e+t), \ \ t\ge 0.
\end{equation}
Such a function $g$ belongs to $C^1$, be convex and non-decreasing, satisfying $g(0)=0$ and $g(t)>0$ for every $t>0$; in order to show that $g$
satisfies (\ref{condition for g-1}) 
we note that 
\begin{equation}\label{t1.2}
g(\lambda t) \leq \lambda^{\mu } g(t) \Longleftrightarrow e+\lambda t  \le (e+t) ^{\lambda ^{\mu-1}}.
\end{equation}
Since $e+\lambda t \le (e+t)^\lambda $ for every $\lambda > 1 $ and every $t>0$, then the right hand side inequality of (\ref{t1.2})
holds true for $\mu =2$ as well; 
it is obvious that $g'(0_+) =1>0$.

\vspace{2mm}

For a function $g$ satisfying the assumptions (${\cal G}_1$) and (${\cal G}_2$), we have the following lemmas:

\begin{lemma}
Consider $g: [0,+\infty)\rightarrow [0,+\infty)$ of class $C^1$, convex, non-decreasing, $g(0)=0$, and satisfying (\ref{condition for g-1}). Then for all $t\ge 0$,
\begin{equation}\label{condition for g-3}
g'(t)t \le \mu g(t).
\end{equation}
\end{lemma}

\begin{proof}
For all $t>0$ and all $\varepsilon>0$, (\ref{condition for g-1}) gives
$$
\frac {g(t+\varepsilon) -g(t)}{\varepsilon} =\frac {g\left( t\left( 1+\frac \varepsilon t  \right)\right)-g(t)}{\varepsilon} \le
\frac {\left[ \left( 1+\frac \varepsilon t\right) ^\mu -1\right]g(t)}{\varepsilon}.
$$
Letting $\varepsilon \rightarrow 0_+ $ in both sides we get
\begin{equation*}\label{d-1}
g'(t_+) t \le \mu g(t)
\end{equation*}
for all $t>0$ since
$$
\lim_{\varepsilon \rightarrow 0_+}\frac {g(t+\varepsilon) -g(t)}{\varepsilon} =g'(t_+),\ \ \  \lim_{\varepsilon \rightarrow 0_+} \frac {\left[ \left( 1+\frac \varepsilon t\right) ^\mu -1\right]}{\varepsilon} =\frac \mu t.
$$
We use the fact $g\in C^1$ and then  (\ref{condition for g-3}) holds true for all $t>0$, and by continuity, for all $t\ge 0$.
\end{proof}

\begin{lemma}\label{new lemma}  Let $g: [0,+\infty)\rightarrow [0,+\infty)$ be of class $C^1$, convex and non-decreasing. Then for any $t_1,t_2\ge 0$,
\begin{equation}\label{new inequality}
g'(t_1)t_2 \le g'(t_1)t_1 +g'(t_2) t_2.
\end{equation}
\end{lemma}
\begin{proof}
Since $g(t)$ be a $C^1$, convex and non-decreasing function, then $g'(t)$ exists and to be non-negative and non-decreasing. If $t_1\le t_2$, then
$g'(t_1)t_2 \le g'(t_2) t_2$; if $t_1>t_2$, then $g'(t_1)t_2 \le g'(t_1)t_1$. In both cases we have (\ref{new inequality}).
\end{proof}

The following lemma  can be found in \cite{Cupini-Marcellini-Mascolo-1}, see also \cite{Cupini-Marcellini-Mascolo-2,Cupini-Marcellini-Mascolo-3}.

\begin{lemma}\label{lemma 2}
Let $v \in W^{1,1} (\Omega)$ and let $g: [0,+\infty) \rightarrow [0,+\infty)$ be of class $C^1$, convex, non-decreasing, non-constant,
$g(0) = 0$ and $g(\lambda t )\le \lambda ^\mu g(t)$ for some $\mu \ge 1$, every $\lambda >1$ and every $t\ge 0$. Suppose that
$g(|Dv|) \in L_{loc} ^p(\Omega)$, then $g(|v|) \in L_{loc} ^{p*} (\Omega)$. 
\end{lemma}

We remark that the papers \cite{Cupini-Marcellini-Mascolo-1,Cupini-Marcellini-Mascolo-2,Cupini-Marcellini-Mascolo-3} provide fruitful ideas in dealing with regularity properties of variational integrals with $g$-growth conditions.

The following iteration lemma is useful and will be repeatedly used in the sequel, see Lemma 7.1 in \cite{G}.

\begin{lemma}\label{Stampacchia Lemma proof need}
Let $\beta, B, C, x_i$ be such that $\beta>1, C>0, B>1, x_i \ge 0$ and
\begin{equation}
x_{i+1}\le CB^{i}x_{i}^{\beta },  \ \ i=0,1,2,\cdots. \label{generalized Stampachia lemma proof need}
\end{equation}
If $x_{0}\le C^{-\frac{1}{\beta -1} } B^{-\frac{1}{\left ( \beta -1 \right )^{2}}}$, then $ x_{i}\le B^{-\frac{i}{\alpha } }x_{0}, i=0,1,2,\cdots$,
so that
$$
\lim_{i \to \infty }x_{i}=0.
$$
\end{lemma}

\section{The first generalization and applications.}
\noindent In this section we give the first generalization of Lemma \ref{Stampacchia Lemma, Gao Zhang and Ma generalized}, that is, (\ref{stampacchia_assumption-2}) is replaced by (\ref{the first condition}),
where the function $g$ is the one introduced in Section 1, and give applications to degenerate elliptic equations of divergence type with $g$-growth conditions.

We prove the following

\begin{theorem}\label{Stampacchia Lemma, Gao Zhang and Ma generalized,generalized}
Let $g:[0,+\infty)\rightarrow [0,+\infty)$ be a function satisfying the assumptions (${\cal G}_1$), (${\cal G}_2$) and (${\cal G}_3$);
let $c , \alpha, \beta, k_0$ be positive constants; let $\theta \ge 0$ satisfies
\begin{equation}\label{condition for L}
\lim _{L\rightarrow +\infty} \frac {L^\theta}{g(L)} =0;
\end{equation}
let
$\varphi: [k_0,+\infty)$ $\rightarrow [0, +\infty)$ be non-increasing
and such that
\begin{equation}\label{stampacchia_assumption generalized-generalization}
\varphi (h) \le \frac {c  h^{\theta \alpha } }{g^\alpha (h-k)} [\varphi (k)]^\beta
\end{equation}
for every $h,k$ with $h>k\ge k_0>0$. It results that:
	
{\bf (i)} if $\beta > 1$ then
\begin{equation*}\label{statement_stampacchia_beta>1,generalized}
\varphi(2L) = 0,
\end{equation*}
where $L$ be a constant satisfying
\begin{equation}\label{the value of L,generalized}
L\ge 2k_0 \quad and \quad \frac{g\left ( L \right ) }{L^{\theta } } \ge c ^{\frac{1}{\alpha } }  [\varphi (k_0)] ^{\frac {\beta -1}{\alpha}}
 2^{\frac 1 {\beta} \left(\mu \beta +\theta +\frac \mu  {\beta -1}\right)};
\end{equation}
	
{\bf(ii)} if $\beta = 1$ then for any $k\ge k_0$,
\begin{equation*} \label{generalied statement_stampacchia_beta=1,generalized}
\varphi (k)\le \varphi(k_0) e ^{1-\left( \frac {k-k_0}{\tilde \tau} \right) ^{1-\theta }},
\end{equation*}
where $\tilde \tau$ is large enough such that
\begin{equation}\label{c2.3 new}
\varphi (k_0+\tilde \tau) \le \frac {\varphi (k_0)}{e}
\end{equation}
and
\begin{equation}\label{tau,generalized}
\tilde \tau \ge \max \left \{ k_{0} ,\left ( \frac{(ce) ^{\frac 1 \alpha} 2^{\frac {(2-\theta ) \theta }{1-\theta }} (1- \theta )}{g'( 0_+)
  }  \right )^{\frac{1}{ 1-\theta } }   \right \} ;
\end{equation}
	
{\bf(iii)} if $0< \beta < 1$ then for any $k\ge k_0$,
\begin{equation} \label{statement_stampacchia_beta<1,generalized}
\varphi (k) \le 2^{\frac { \mu\alpha (1-\theta)\left ( 2-\beta  \right ) }{(1-\beta)^2 }} \left\{
\left(\frac {c  2 ^{\theta \alpha}}{g'(0_+) ^{\theta \alpha}}\right) ^\frac {1}{1-\beta} +g( k_{0})^{\frac {(1-\theta)\alpha} {1-\beta}}
 \varphi(k_0) \right\} \left(\frac{1}{g\left ( k \right ) } \right) ^{\frac{(1-\theta)\alpha} {1-\beta}}.
\end{equation}
\end{theorem}

\begin{remark}
We remark that $\theta =1$ satisfies (\ref{condition for L}) for the case $g(t) =t\ln (e+t)$.
\end{remark}

\begin{remark}
The second inequality of (\ref{the value of L,generalized}) holds true for $L $ sufficiently large because of (\ref{condition for L}).
\end{remark}

\begin{remark}
(\ref{c2.3 new}) holds true for $\tilde \tau$ sufficiently large. In fact, we take $k=k_0$ in (\ref{stampacchia_assumption generalized-generalization}) and we have, for all $h>k_0>0$,
\begin{equation}\label{22.1}
\varphi (h) \le \frac {c  h^{\theta \alpha } }{g^\alpha (h-k_0)} [\varphi (k_0)]^\beta.
\end{equation}
Since by (\ref{condition for L})
$$
\lim_{h\to +\infty} \frac {h^\theta}{g(h-k_0 )} \le 2 ^{\max \{\theta, 1\} -1} \lim _{h \to +\infty} \frac {(h-k_0) ^\theta +k_0^\theta}{g(h-k_0)} =0,
$$
then from (\ref{22.1}) one gets
$$
\lim _{h\to +\infty} \varphi (h) =0.
$$
\end{remark}

\noindent{\it Proof of Theorem \ref{Stampacchia Lemma, Gao Zhang and Ma generalized,generalized}.} (i) For $\beta>1$, we fix $L>k_0$ and choose levels
$$
t_i =2L\left ( 1-2^{-i-1}  \right ) , \ \ i=0,1,2,\cdots,
$$
It is obvious that $k_{0}< L\le t_{i}< 2L $ and $\left \{ t_{i}  \right \} $ be an increasing sequence. We choose in
 (\ref{stampacchia_assumption generalized-generalization})
$$
k=t_{i} , \ h=t_{i+1},\ x_{i}=\varphi \left ( t_{i}  \right ),\  x_{i+1}=\varphi \left ( t_{i+1}  \right ),
$$
and noticing that $h-k=t_{i+1}-t_{i}=L2^{-i-1}$, we have
\begin{equation}\label{c2.5}
x_{i+1}\le \frac{c \left ( 2L\left ( 1-2^{-i-2} \right )  \right ) ^{\theta \alpha } }{g^{\alpha }\left ( L2^{-i-1} \right )  }
 x_{i}^{\beta }\le \frac{c  2^{\theta \alpha}L ^{\theta \alpha}}{g^{\alpha }\left ( L2^{-i-1} \right )  } x_{i}^{\beta }  , \ \ i=0,1,2,\cdots.
\end{equation}
By (\ref{condition for g-1}),
\begin{equation}\label{c2.6}
g\left ( L \right )=g\left ( 2^{i+1} L2^{-i-1}  \right )\le 2^{\left ( i+1 \right )\mu  } g(L2^{-i-1}).
\end{equation}
(\ref{c2.5}) and (\ref{c2.6}) imply
$$
x_{i+1}\le \frac{c  2^{\left ( \theta +  \mu  \right )\alpha  }L^{\theta \alpha } 2^{i\mu \alpha }  }{g^{\alpha }\left ( L \right )  }x_{i}^{\beta }, \ \ i=0,1,2, \cdots.
$$
Thus (\ref{generalized Stampachia lemma proof need}) holds true with (we keep in mind that $\beta>1$)
$$
C=\frac{c  2^{\left ( \theta +\mu  \right )\alpha  }L^{\theta \alpha }  }{g^{\alpha }\left ( L \right )  }   \quad \mbox { and  } \quad B=2^{\mu\alpha}.$$
Note that the constant $C$ above is finite since  (\ref{condition for L}). We get from Lemma \ref{Stampacchia Lemma proof need} that
\begin{equation}\label{Stampacchia Lemma proof need result}
\lim_{i \to \infty} x_{i}=0
\end{equation}
provided that
\begin{equation}\label{value L}
\begin{array}{llll}
\displaystyle x_{0}=\varphi \left ( t_{0}  \right )= \varphi (L) & \le &\displaystyle \left ( \frac{c  2^{\left ( \theta +\mu  \right )\alpha}
 L^{\theta \alpha  } }{g^{\alpha }\left ( L \right )  }  \right )^{-\frac{1}{\beta -1} } \left ( 2^{\mu \alpha }  \right )^{-\frac{1}{\left ( \beta -1 \right )^{2}  } } \\
&=&\displaystyle \frac{c ^{-\frac{1}{\beta -1} } 2^{-\frac{\alpha }{\beta -1} \left ( \theta +\mu +\frac{\mu }{\beta -1}  \right ) }
L^{-\frac{\theta \alpha }{\beta -1} } }{g^{-\frac{\alpha }{\beta -1} } \left ( L \right ) }.
\end{array}
\end{equation}
Note that (\ref{Stampacchia Lemma proof need result}) implies
$$
\varphi \left ( 2L \right )=0.
$$
Let us check (\ref{value L}) and determine the value of $L$. 
In (\ref{stampacchia_assumption generalized-generalization}), we take $k=k_0$ and $h=L\ge 2k_0$ (which is equivalent to $L-k_0 \ge \frac{L}{2}$) and we have
$$
\varphi \left ( L \right )\le \frac{c L^{\theta \alpha } }{g^{\alpha } \left ( L-k_0 \right ) } \left [ \varphi \left ( k_0 \right )
 \right ] ^{\beta }\le \frac{c L^{\theta \alpha }2^{\mu \alpha }  }{g^{\alpha }\left ( L \right )  }\left [ \varphi \left ( k_0 \right )  \right ]^{\beta }   ,
$$
where we have used  the fact
$$
g^{\alpha }\left ( L-k_0 \right )\ge g^{\alpha }\left ( \frac{L}{2} \right )\ge 2^{-\mu \alpha }g^{\alpha }\left ( L \right ).
$$
(\ref{value L}) would be satisfied if $L \ge 2k_0$ and
$$
\frac{c L^{\theta \alpha }2^{\mu \alpha }  }{g^{\alpha }\left ( L \right )  }\left [ \varphi \left ( k_0 \right )  \right ]^{\beta }\le
\frac{c ^{-\frac{1}{\beta -1} } 2^{-\frac{\alpha }{\beta -1} \left ( \theta +\mu +\frac{\mu }{\beta -1}  \right ) } L^{-\frac{\theta \alpha }
{\beta -1} } }{g^{-\frac{\alpha }{\beta -1} } (L)}.
$$
The above inequality is equivalent to the second inequality in (\ref{the value of L,generalized}).

\vspace{3mm}

(ii) Let $\beta=1$ and assume that $\tilde \tau$ satisfies (\ref{c2.3 new}) and (\ref{tau,generalized}). Let
$$
k_{s}=k_{0}+\tilde \tau s^{\frac{1}{1-\theta } }, \ \ s=0,1,2,\cdots,
$$
then $\left \{ k_{s}  \right \}$ is an increasing sequence and
$$
k_{s+1}-k_{s}=\tilde \tau \left ( \left ( s+1 \right )^{\frac{1}{1-\theta } } -s^{\frac{1}{1-\theta } } \right ) .
$$
We use Taylor's formula to get
$$
k_{s+1}-k_{s}=\tilde  \tau \left( \frac{1}{1-\theta}s^{\frac{\theta}{1-\theta} }+\frac{\theta }{2!\left ( 1-\theta  \right )^{2}  }\xi
^{\frac{2\theta -1}{1-\theta} } \right)\ge \frac{\tilde \tau }{1-\theta}s^{\frac{\theta}{1-\theta} } ,
$$
where $\xi$ lies in the open interval $\left (s,s+1  \right ) $. In (\ref{stampacchia_assumption generalized-generalization}) let us take
$\beta=1, k=k_s$ and $h=k_{s+1}$, we use the above inequality and obtain, for $s\ge 1$,
\begin{equation}\label{Recursive condition}
\varphi \left ( k_{s+1}  \right ) \le \frac{c \left( k_{0}+\tilde \tau\left ( s+1 \right )^{\frac{1}{1-\theta} }   \right )^{\theta\alpha}
 }{g^{\alpha } \left ( \frac{\tilde \tau}{1-\theta} s^{\frac{\theta}{1-\theta } }  \right ) }\varphi \left ( k_s \right )  \le \frac{c
 \left ( k_{0}+\tilde \tau\left ( 2s \right )^{\frac{1}{1-\theta} }   \right ) ^{\theta\alpha}  }{g^{\alpha } \left ( \frac{\tilde \tau}{1-\theta}
  s^{\frac{\theta}{1-\theta } }  \right ) }\varphi \left ( k_s \right ) .
\end{equation}
We use the fact $g(0)=0$ and Lagrange mean value theorem to obtain
\begin{equation}\label{Lagrange inequality}
g\left ( \frac{\tilde \tau}{1-\theta} s^{\frac{\theta}{1-\theta} }  \right )=g\left ( \frac{\tilde \tau}{1-\theta} s^{\frac{\theta}{1-\theta} }
 \right ) -g\left ( 0 \right ) =g'\left ( \zeta  \right )\frac{\tilde \tau}{1-\theta} s^{\frac{\theta}{1-\theta} }\ge g'\left (0_+ \right )
 \frac{\tilde \tau}{1-\theta} s^{\frac{\theta}{1-\theta} } ,
\end{equation}
where $ \zeta$ lies in the open interval $\left ( 0,\frac{\tilde \tau}{1-\theta} s^{\frac{\theta}{1-\theta} }  \right )  $. Substituting
(\ref{Lagrange inequality}) into (\ref{Recursive condition}) and noticing, for $s\ge 1$,
$$
k_{0}\le \tilde \tau<\tilde \tau\left ( 2s \right )^{\frac{1}{1-\theta} }
$$
then for $s\ge 1$,
\begin{equation}\label{c2.8}
\begin{array}{llll}
\varphi (k_{s+1}) &\le &\displaystyle \frac {c  \left(k_0+ \tilde \tau  (2s) ^{\frac 1 {1-\theta}}\right)^{\theta \alpha} }{g'(0_+)^\alpha
  \left(\frac {\tilde \tau }{1-\theta }\right) ^\alpha  s ^{\frac {\theta \alpha}{1-\theta}}} \varphi (k_s)\le \frac {c 2^{\theta \alpha}
  \left(\tilde \tau  (2s) ^{\frac 1 {1-\theta}}\right)^{\theta \alpha} }{g'(0_+)^\alpha  \left(\frac {\tilde \tau }{1-\theta }\right) ^\alpha
   s ^{\frac {\theta \alpha}{1-\theta}}} \varphi(k_s)\\
&=&\displaystyle \frac {c 2^{\theta \alpha} \left(\tilde \tau  2^{\frac 1 {1-\theta}}\right)^{\theta \alpha} }{g'(0_+)^\alpha
 \left(\frac {\tilde \tau }{1-\theta }\right) ^\alpha  } \varphi(k_s) =\frac{c \left ( 2^{\frac{2-\theta}{1-\theta} }\tilde \tau  \right )^{\theta\alpha}
  }{g'\left ( 0_+\right )^{\alpha } \left ( \frac{\tilde \tau}{1-\theta}  \right )^{\alpha}  }\varphi \left ( k_{s}  \right ) \\
&\le&\displaystyle   \frac 1 e \varphi(k_s),
\end{array}
\end{equation}
where in the last inequality we have used (\ref{tau,generalized}), which ensures
$$
\frac{c \left ( 2^{\frac{2-\theta}{1-\theta} }\tilde \tau  \right )^{\theta\alpha}
  }{g'\left ( 0_+\right )^{\alpha } \left ( \frac{\tilde \tau}{1-\theta}  \right )^{\alpha}  }\le \frac{1}{e} .
$$
(\ref{c2.3 new}) ensures (\ref{c2.8}) also holds true for $s=0$ as well.
For any $k \ge k_0$, there exists $s\in \mathbb {N}^{+}$ such that
$$
k_{0}+\tilde \tau\left ( s-1 \right )^{\frac{1}{1-\theta} }  \le k<k_{0}+\tilde \tau s^{\frac{1}{1-\theta} } .
$$
Thus, considering $	\varphi \left ( k \right )$ is non-increasing, one has
$$
\varphi \left ( k \right ) \le \varphi \left ( 	k_{0}+\tilde \tau\left ( s-1 \right )^{\frac{1}{1-\theta} }  \right ) =\varphi \left ( k_{s-1}
\right )\le e^{1-s}\varphi \left ( k_{0}  \right )\le \varphi \left ( k_{0}  \right )e^{1-\left ( \frac{k-k_0}{\tilde \tau}  \right )^{1-\theta}}.
$$

\vspace{3mm}

(iii) For $0<\beta <1$ we use Lagrange mean value theorem again and we have, for $k\ge k_0>0$,
$$
g\left ( k \right )=g\left ( k \right )-g\left ( 0 \right ) =g'\left ( \varsigma  \right ) k\ge g'\left ( 0_+\right ) k ,
$$
where $\varsigma $ lies in the open interval $\left (0,k  \right ) $. We take $h=2k$ in (\ref{stampacchia_assumption generalized-generalization})
and we get, for every $k \ge k_0>0 $,
\begin{equation}\label{22.2}
\varphi \left ( 2k \right ) \le \frac{c \left ( 2k\right ) ^{\theta \alpha } }{g^{\alpha }\left ( k \right )   }\left [ \varphi
\left ( k \right )  \right ]^{\beta }\le \frac{c 2^{\theta \alpha } }{ g'\left ( 0_+ \right )^{\theta \alpha }g\left ( k \right )^
{\left ( 1-\theta \right )\alpha }} \left [ \varphi \left ( k \right )  \right ]^{\beta }=\frac {c_3}{g^{\tilde \alpha}(k)} [\varphi (k)]^\beta,
\end{equation}
where
\begin{equation}\label{22.3}
c_3= \frac {c 2^{\theta \alpha}}{g'(0_+) ^{\theta \alpha}}, \ \ \tilde \alpha =(1-\theta) \alpha.
\end{equation}
If we introduce a new function $\psi (k): [k_0,+\infty) \rightarrow [0,+\infty)$ such that
$$
\varphi (k) =\psi (k) \frac {c_3^{\frac 1 {1-\beta}}}{g^{\frac {\tilde \alpha}{1-\beta}}(k)},
$$
then we use (\ref{22.2}) and (\ref{condition for g-1})  and we get
\begin{equation}\label{2.8}
\psi (2k)\le \frac {g^{\frac {\tilde \alpha} {1-\beta}}(2k)}{g^{\frac {\tilde \alpha} {1-\beta}}(k)} \psi ^\beta (k) \le 2 ^{\frac
{\mu {\tilde \alpha}}{1-\beta}} \psi ^\beta (k), \ \ k\ge k_0.
\end{equation}
For any $k\ge k_0$, one can find a natural number $s$ such that
$$
2^{s-1} k_0 \le k < 2 ^s k_0,
$$
for such a $k$ one has
\begin{equation}\label{2.9}
\varphi (k)\le \varphi (2^{s-1}k_0)  =\psi (2^{s-1}k_0) \frac {c_3 ^{\frac 1 {1-\beta}}} {g^{\frac {\tilde \alpha} {1-\beta}} (2^{s-1}k_0)} .
\end{equation}
(\ref{2.8}) implies, for $s>1$,
\begin{equation}\label{2.10}
\begin{array}{llll}
 \displaystyle \psi (2^{s-1}k_0) &\le & \displaystyle  2 ^{\frac {\mu {\tilde \alpha}}{1-\beta } \sum \limits_{i=0}^{s-2}\beta ^i}  \psi ^{\beta ^{s-1}} (k_0)
\le 2 ^{\frac {\mu {\tilde \alpha}}{(1-\beta )^2}}(1+\psi (k_0)) \\
 & \le & \displaystyle  2 ^{\frac {\mu {\tilde \alpha}}{(1-\beta )^2}} \left(1+\varphi (k_0) g^{\frac {\tilde \alpha} {1-\beta}} (k_0) c_3^{\frac 1 {\beta -1}}\right).
\end{array}
\end{equation}
The above inequality holds true for $s=1$ as well. Since
$$
g (2 ^s k_0) =g \left(2 (2^{s-1}k_0) \right) \le 2 ^\mu g(2^{s-1}k_0),
$$
then
\begin{equation}\label{2.11}
g(2^{s-1}k _0) \ge 2^{-\mu } g(2^s k_0) \ge 2^{-\mu} g(k).
\end{equation}
Substituting (\ref{2.10}) and (\ref{2.11}) into (\ref{2.9}) we arrive at
\begin{equation*}\label{2.12}
\varphi (k) \le 2 ^{\frac {\mu {\tilde \alpha}(2-\beta)}{(1-\beta )^2}} \left( c_3 ^{\frac 1 {1-\beta }}+\varphi (k_0) g^{\frac
{\tilde \alpha} {1-\beta}} (k_0)\right) \left(\frac 1 {g(k)}\right) ^{\frac {\tilde \alpha} {1-\beta}}.
\end{equation*}
This inequality is equivalent to (\ref{statement_stampacchia_beta<1,generalized}) due to the definition of $c_3$ and $\tilde \alpha$ in (\ref{22.3}), completing the proof of Theorem \ref{Stampacchia Lemma, Gao Zhang and Ma generalized,generalized}.
\qed

\vspace{3mm}

If one takes $\theta =0$ in (\ref{stampacchia_assumption generalized-generalization}), then  the following corollary of Theorem \ref{Stampacchia Lemma, Gao Zhang and Ma generalized,generalized} is obvious.

\begin{corollary}\label{Stampacchia Lemma, Gao Zhang and Ma generalized,generalized-new}
Let $g:[0,+\infty)
 \rightarrow [0,+\infty)$ be a function satisfying the assumptions (${\cal G}_1$), (${\cal G}_2$) and (${\cal G}_3$).
 Let $c , \alpha, \beta, k_0$ be positive constants; let
$\varphi: [k_0,+\infty)$ $\rightarrow [0, +\infty)$ be non-increasing
and such that
\begin{equation}\label{stampacchia_assumption generalized-generalization-new}
\varphi (h) \le \frac {c  }{g^\alpha (h-k)} [\varphi (k)]^\beta
\end{equation}
for every $h,k$ with $h>k\ge k_0>0$. It results that:
	
{\bf (i)} if $\beta > 1$ then
\begin{equation}\label{statement_stampacchia_beta>1,generalized-new}
\varphi(2L) = 0,
\end{equation}
where $L$ be a constant satisfying
\begin{equation*}\label{the value of L,generalized-new}
L\ge 2k_0 \quad and \quad g\left ( L \right ) \ge c ^{\frac{1}{\alpha } }  [\varphi (k_0)] ^{\frac {\beta -1}{\alpha}}
 2^{\left(\mu +\frac \mu  {\beta(\beta -1) }\right)};
\end{equation*}
	
{\bf(ii)} if $\beta = 1$ then for any $k\ge k_0$,
\begin{equation} \label{generalied statement_stampacchia_beta=1,generalized-new}
\varphi (k)\le \varphi(k_0) e ^{1-\frac {k-k_0}{\tilde \tau}},
\end{equation}
where $\tilde \tau$ is large enough such that
\begin{equation}\label{c2.3 new new}
\varphi (k_0+\tilde \tau) \le \frac {\varphi (k_0)}{e}
\end{equation}
and
\begin{equation}\label{tau,generalized-new new}
\tilde \tau \ge \max \left \{ k_{0} ,\frac{(c e)^{\frac 1 \alpha} }{g'( 0_+)
}  \right \} ;
\end{equation}
	
{\bf(iii)} if $0< \beta < 1$ then for any $k\ge k_0$,
\begin{equation*} \label{statement_stampacchia_beta<1,generalized-new}
\varphi (k) \le 2^{\frac { \mu\alpha \left ( 2-\beta  \right ) }{(1-\beta)^2 }} \left\{
c ^\frac {1}{1-\beta} +g( k_{0})^{\frac {\alpha} {1-\beta}}
 \varphi(k_0) \right\} \left(\frac{1}{g\left ( k \right ) } \right) ^{\frac{\alpha} {1-\beta}}.
\end{equation*}
\end{corollary}

\vspace{3mm}

In the remaining part of this section we shall restrict ourselves to the case $\theta =0$ in (\ref{stampacchia_assumption generalized-generalization}).
We now turn our attention to (\ref{stampacchia_assumption generalized-generalization-new}). Let us consider the inequality
\begin{equation}\label{new-new}
\varphi  ( 2k  )\le \frac{\tilde c}{g^{\alpha } ( k  )  } [ \varphi  ( k  )   ]^{\beta }, \ \ \   \forall k\ge k_0 >0.
\end{equation}
Let us compare (\ref{stampacchia_assumption generalized-generalization-new}) with (\ref{new-new}): it is obvious that if one takes $h=2k$ in
(\ref{stampacchia_assumption generalized-generalization-new}), then one gets (\ref{new-new}) with $\tilde c=c $. We ask the following question: is
 (\ref{new-new}) weaker than (\ref{stampacchia_assumption generalized-generalization-new}) for all cases of $\beta $? The answer is: it depends.
 For different cases of the value of $\beta$, we have different answers. 

The following theorem says that, in case of $0<\beta <1$, the two assumptions (\ref{stampacchia_assumption generalized-generalization-new}) and (\ref{new-new}) are equivalent.

\begin{theorem}\label{remark 1}
Let $g:[0,+\infty)  \rightarrow [0,+\infty)$ be a function satisfying the assumptions (${\cal G}_1$), (${\cal G}_2$) and (${\cal G}_3$).
Let $\varphi : [k_{0},+  \infty )\to [0,+  \infty )$ be non-increasing and $\alpha \in \left ( 0,+  \infty  \right ) $, $\beta  \in \left ( 0,1 \right ) $ be constants. Then
\begin{equation*}
	(\ref{stampacchia_assumption generalized-generalization-new})\Leftrightarrow (\ref{new-new}).
\end{equation*}
\end{theorem}
\begin{proof} ``$\Rightarrow$".  We take $h=2k$ in (\ref{stampacchia_assumption generalized-generalization-new}) and we get (\ref{new-new}) with $\tilde c=c$.

``$\Leftarrow$". Assume (\ref{new-new}). Let us consider $h> k\ge k_{0}$. We divide the proof into two cases: $2^{n+1}k\ge h>2^{n}k$ for some integer $n\ge 1$ and $2k\ge h>k$.

{\bf Case 1}: $2^{n+1}k\ge h>2^{n}k$ for some integer $n\ge 1$. Since $\varphi$ decreases, we have
\begin{equation}\label{2.14}
\varphi  ( h  )\le \varphi  ( 2^{n}k   )=\varphi  ( 2 ( 2^{n-1}k  )  ),
\end{equation}
here we keep in mind that $n\ge 1$, so $2^{n-1}k\ge k\ge k_{0}$ and we use (\ref{new-new}) with $2^{n-1}k$ in place of $k$:
\begin{equation}\label{2.15}
\varphi  ( 2 ( 2^{n-1}k ))\le \frac{\tilde c}{g^{\alpha}(2^{n-1}k)}[\varphi(2^{n-1}k)]^{\beta }.
\end{equation}
Since $2^{n-1}k\ge k$, we use the monotonicity of $\varphi$ to obtain $\varphi  ( 2^{n-1}k  ) \le \varphi  ( k  )$, namely
\begin{equation}\label{New-new 1}
 [ \varphi  ( 2^{n-1}k  )  ]^{\beta }   \le  [ \varphi  ( k  )   ]^{\beta }.
\end{equation}
Since  $2^{n+1}k\ge h$, we have $ ( 2^{n+1} -1  )k\ge h-k$, then
\begin{equation*}
2^{n-1}k=\frac{2^{n+1}k}{4}\ge \frac{ ( 2^{n+1}-1   ) k}{4}\ge \frac{h-k}{4},
\end{equation*}
and using the monotonicity of $g$ and (\ref{condition for g-1}) again, we get
\begin{equation}\label{2.17}
g^{\alpha } ( 2^{n-1} k  )\ge g^{\alpha } \left( \frac{h-k}{4}   \right)\ge \frac{1}{4^{\mu \alpha } } g^{\alpha } ( h-k  ).
\end{equation}
Combining (\ref{2.14}), (\ref{2.15}), (\ref{New-new 1}) and (\ref{2.17}) we arrive at
\begin{equation*}
\varphi (h)
\le \frac{\tilde c 4^{\mu \alpha }}{g^{\alpha }( h-k )}[\varphi ( k )]^{\beta },
\end{equation*}
which shows that (\ref{stampacchia_assumption generalized-generalization-new}) holds true for $c =\tilde c 4 ^{\mu \alpha}$.

{\bf Case 2: $2k\ge h>k$}.
By Corollary \ref{Stampacchia Lemma, Gao Zhang and Ma generalized,generalized-new}-(iii),  (\ref{stampacchia_assumption generalized-generalization-new}) implies
$$
\varphi (k) \le \bar c \left(\frac{1}{g\left ( k \right ) } \right) ^{\frac{\alpha} {1-\beta}},
$$
where
\begin{equation}\label{c_2-new}
\bar c = 2^{\frac { \mu\alpha \left ( 2-\beta  \right ) }{(1-\beta)^2 }} \left\{
c ^\frac {1}{1-\beta} +g( k_{0})^{\frac {\alpha} {1-\beta}}
 \varphi(k_0) \right\} .
\end{equation}
%
Since $\varphi$ decreases we have 
\begin{equation*}
\varphi ( h ) \le \varphi   ( k   ) =  [ \varphi   ( k   )    ]^{1-\beta } [ \varphi   ( k   )    ]^{\beta }
\le \frac {\bar c  ^{1-\beta}}{ g^{\alpha}(k) } [\varphi (k)]^\beta.
\end{equation*}
Since $2k\ge h$ we get $k\ge h-k$, and due to the monotonicity of $g$, we obtain $g^{\alpha}  ( k   )\ge  g^{\alpha}(h-k)$, then
\begin{equation*}
\varphi  ( h   )\le \frac{\bar c^{1-\beta }  } {g^{\alpha}  ( h-k   )  }  [ \varphi   ( k   )    ]^{\beta } .
\end{equation*}
In both cases we have obtained (\ref{stampacchia_assumption generalized-generalization-new}) with $ c =\max  \left\{\tilde c4^{\mu \alpha },
\bar c^{1-\beta } \right  \} $ with $\bar c$ be as in (\ref{c_2-new}).
\end{proof}

Let us now consider the remaining two cases $\beta =1$ and $\beta >1$.
The following theorems say that, for these two cases of $\beta$, the assumptions (\ref{stampacchia_assumption generalized-generalization-new}) and (\ref{new-new}) are not equivalent.

\begin{theorem}\label{remark 2}
Let $g:[0,+\infty)  \rightarrow [0,+\infty)$ be a function satisfying the assumptions (${\cal G}_1$), (${\cal G}_2$)  and (${\cal G}_3$).
Let $\varphi : [k_{0},+  \infty )\to [0,+  \infty )$ be non-increasing. Let $\alpha \in \left ( 0,+  \infty  \right ) $ be a constant and $\beta =1$. Then
\begin{equation*}
	(\ref{stampacchia_assumption generalized-generalization-new})\not\Leftarrow (\ref{new-new}).
\end{equation*}
More precisely, the function
\begin{equation}\label{phi}
\varphi (k)=e ^{-(\ln k)^2}, \ \ k\ge 1,
\end{equation}
verifies (\ref{new-new}) with $k_0=1$, $\beta =1$, $\tilde c=2 ^{-\ln 2}$, $g(k)=k^{\ln 2}$ and $\alpha =2$,  but it does not
satisfy (\ref{stampacchia_assumption generalized-generalization-new}) with $\beta =1$, for any choice of the two constants $\alpha >0$, $c>0$ and any choice of the function $g(t)$.
\end{theorem}

\begin{proof} We take $\varphi $ be as in (\ref{phi}) and we have
$$
\begin{array}{llll}
&\displaystyle \varphi (2k) =e ^{-(\ln (2k))^2} =e ^{-(\ln 2 +\ln k)^2} =e ^{-[(\ln 2)^2 +2 \ln 2 \ln k +(\ln k)^2]}\\
=&\displaystyle e ^{-\ln 2 (\ln 2 +2\ln k)} e ^{-(\ln k)^2}= e ^{\ln (2k^2) ^{-\ln 2}} \varphi (k) =\left(\frac 1 {2k^2} \right)^{\ln 2}
 \varphi (k) =\frac {2^{-\ln 2}}{(k ^{\ln 2})^2 }\varphi(k).
\end{array}
$$
This shows that (\ref{new-new}) holds true with $k_0=1$, $\beta =1$, $\tilde c=2 ^{-\ln 2}$, $g(k)=k^{\ln 2}$ and $\alpha =2 $ .

Now we shall show that (\ref{stampacchia_assumption generalized-generalization-new}) does not hold true with $\beta =1$:
by contradiction, if (\ref{stampacchia_assumption generalized-generalization-new})
would hold true with $\beta =1$, then Corollary \ref{Stampacchia Lemma, Gao Zhang and Ma generalized,generalized-new}-(ii)
would guarantees (\ref{generalied statement_stampacchia_beta=1,generalized-new}), then
$$
\varphi (k) \le \hat c e ^{-\lambda k}
$$
where
$$
\hat c= \varphi (k_0) e ^{1+\frac {k_0}{\tilde \tau}}  \ \mbox { and } \ \lambda =\frac {1}{\tilde \tau}.
$$
That is
$$
e ^{-(\ln k)^2} \le  \hat c e ^{-\lambda k} ,
$$
this is equivalent to
$$
e ^{\lambda k -(\ln k)^2} \le \hat c,
$$
but this is false for sufficient large $k$ since the left hand side approaches to $+\infty$ as $k\rightarrow +\infty$.
\end{proof}

\begin{theorem}\label{remark 3}
Let $g:[0,+\infty)  \rightarrow [0,+\infty)$ be a function satisfying the assumptions (${\cal G}_1$), (${\cal G}_2$)  and (${\cal G}_3$).
Let $\varphi : [k_{0},+  \infty )\to [0,+  \infty )$ be non-increasing. Let $\alpha \in \left ( 0,+  \infty  \right ) $ and $\beta >1$ be constants. Then
\begin{equation*}
	(\ref{stampacchia_assumption generalized-generalization-new})\not\Leftarrow (\ref{new-new}).
\end{equation*}
More precisely, the function
\begin{equation}\label{phi-2}
\varphi (k)=e ^{-k}, \ \ k\ge k_0,
\end{equation}
verifies (\ref{new-new}) with any $\alpha >0$, a suitable $k_0= k_0(\alpha)$, $\beta=\frac 3 2$, $\tilde c=1$ and $g(k)=k^2$,
 but it does not satisfy (\ref{stampacchia_assumption generalized-generalization-new}) for any choice of the three constants $\beta >1$, $\alpha >0$,
  $c >0$ and any choice of the function $g(t)$.
\end{theorem}

\begin{proof}
Let us take $\varphi $ as in (\ref{phi-2}), then
\begin{equation}\label{2.22}
\varphi (2k) =e ^{-2k} =e ^{-\frac k 2 }e ^{-\frac {3k}{2}} =e ^{-\frac k 2 } [\varphi (k)] ^{\frac 3 2 }=\frac
 {k ^{2\alpha}}{e ^{\frac k 2}} \frac {1}{ (k^2)^\alpha}[\varphi (k)] ^{\frac 3 2 }.
\end{equation}
Note that there exists $k_0=k_0(\alpha)\ge 1$ such that
$$
\frac {k ^{2\alpha}}{e ^{\frac k 2}} \le 1, \ \ \forall k\ge k_0,
$$
then (\ref{2.22}) yields
$$
\varphi (2k) \le  \frac {1}{ (k^2)^\alpha}[\varphi (k)] ^{\frac 3 2 }, \ \ \forall k\ge k_0,
$$
so that $\varphi$ verifies (\ref{new-new}) with any $\alpha >0$, with a suitable $k_0=k_0(\alpha)\ge 1$, $\beta=\frac 3 2$, $\tilde c=1$, and $g(k)=k^2$.

We claim that such a $\varphi$ does not satisfy (\ref{stampacchia_assumption generalized-generalization-new})
for any choice of the constants $\beta >1$, $\alpha >0$, $c >0$, $k_0>0$, and for any choice of the function $g(k)$. Indeed, if
 such a $\varphi$ would satisfy (\ref{stampacchia_assumption generalized-generalization-new}), then Corollary \ref{Stampacchia Lemma, Gao Zhang and Ma generalized,generalized-new}-(i) would imply
(\ref{statement_stampacchia_beta>1,generalized-new}), i.e.,
$$
\varphi(2L)=0
$$
for a suitable $L\ge 0$: this gives a contradiction since $\varphi (k)>0$ for every $k>0$.
\end{proof}


We now give an application of Corollary \ref{Stampacchia Lemma, Gao Zhang and Ma generalized,generalized-new}
 to regularity properties of variational integrals with nonstandard growth conditions. Let us consider the variational integral
\begin{equation}\label{variational integral}
{\cal J} (u;\Omega) =\int_\Omega f(x,Du(x))dx,
\end{equation}
where $\Omega$ is a bounded open subset of $\mathbb R^n$, $n\ge 2$, $u: \Omega\rightarrow \mathbb R$ and $f:\Omega \times
\mathbb R^n \rightarrow \mathbb R$ is a Carath\'eodory
function, that is, $x\mapsto f(x,z)$ is measurable and $z\mapsto f(x,z)$ is continuous.

We assume $g^ p$ growth from below: there exist constants $p\in (1,n)$ and $\nu \in (0,+\infty)$, and a function $a(x):
\Omega \rightarrow [0,+\infty) $ such that
\begin{equation}\label{nonstandard growth condition}
\nu g^p (|z|) -a(x) \le f(x,z)
\end{equation}
for almost all $x\in \Omega$ and all $z\in \mathbb R^n$, where $g(t)$ is the function introduced in Section 1.
We fix a boundary datum $u_* :\Omega \rightarrow \mathbb R$ such that $u_* \in W^{1,1} (\Omega)$
and
\begin{equation}\label{L1 integrable}
 f(x, Du_*(x)) \in L^1(\Omega).
\end{equation}
The set of competing functions for the variational integral (\ref{variational integral}) is
$$
{\cal C} = \left\{w \in u_* +W_0^{1,1} (\Omega) \mbox { such that } 
f(x, Dw(x)) \in L^1(\Omega)\right\}.
$$
A function $u:\Omega \rightarrow \mathbb R$ is a minimizer for (\ref{variational integral}) if $u\in \cal C$ and it verifies, for all $w\in \cal C$,
\begin{equation}\label{minimallity}
{\cal J} (u;\Omega) \le {\cal J} (w;\Omega).
\end{equation}

We deal with regularity of minimizers of the variational integral (\ref{variational integral}). Related results have
been obtained in \cite{Leonetti-Petricca}, where the special case $g(t)=t$ is dealt with. Now we consider a  more general case:
there is a function $g(t)$ in (\ref{nonstandard growth condition}). We ask the following question:
if the boundary datum $u_*$ makes the density $f(x,Du_*(x))$ more integrable than (\ref{L1 integrable}) requires, does the minimizer
$u$ enjoy higher integrability?
The answer is positive and we now prove the following

\begin{theorem}\label{theorem 3.1}
Let $g$ be a function satisfying the assumptions (${\cal G}_1$), (${\cal G}_2$) and (${\cal G}_3$). Assume that $a(x) ,f(x,Du_*(x))
\in L^\sigma (\Omega)$ where $\sigma \in (1,+\infty)$. If $u\in \cal C$
minimizes the variational integral (\ref{variational integral}) under (\ref{nonstandard growth condition}), then
\begin{equation}\label{3.5-new}
\begin{array}{llll}
&\displaystyle \sigma >\frac n p \Rightarrow u-u_*\in L^\infty (\Omega); \\
&\displaystyle  \sigma =\frac n p \Rightarrow  \exists \tau (n,p,\mu,\nu,\|f(x,Du_*)+a\|_{L^\sigma (\Omega)}) >0
\mbox { s.t. } e ^{\tau |u-u_*|} \in L^1 (\Omega); \\
&\displaystyle \sigma <\frac n p \Rightarrow g(|u-u_*|) \in L_w ^{\frac {np\sigma }{n-p\sigma }} (\Omega).
\end{array}
\end{equation}
Note that $\frac {np\sigma }{n-p\sigma } >p*= \frac {np}{n-p}$ and $u-u_*\in L^\infty (\Omega)$ is equivalent to $g(u-u_*) \in L^\infty (\Omega)$.
\end{theorem}

In (\ref{3.5-new}) $L_{w} ^m (\Omega)$ is the Marcinkiewicz space (see Definition 3.8 in \cite{Boccardo-Croce}), which consists of all measurable functions $f$ on $\Omega$ with the following property: there exists a constant $\gamma$ such that
\begin{equation}\label{weak space}
|\{x\in \Omega: |f|>\lambda \}| \le \frac {\gamma}{ \lambda ^m}, \ \ \forall \lambda >0,
\end{equation}
where $|E|$ is the Lebesgue measure of the set $E$. The norm of $f\in L_{w} ^m (\Omega)$ is defined by
$$
\|f\|_{L_{weak}^m (\Omega)} ^m =\inf \left\{ \gamma>0, (\ref{weak space}) \mbox { holds} \right\}.
$$
A H\"older inequality holds true for $f\in L_{w}^m (\Omega)$, $m>1$: there exists $B=B( \|f\| _{L_{w} ^m (\Omega)},m)>0$ such that for every measurable subset $E\subset \Omega$,
\begin{equation}\label{Holder}
\int_E |f| dx \le B |E| ^{1-\frac 1 m },
\end{equation}
see Proposition 3.13 in \cite{Boccardo-Croce}. For some basic properties of functions in $L_w^m (\Omega)$, we refer the reader to \cite{Boccardo-Croce}.

We remark that, in the special case $g(t)=t$, Theorem \ref{theorem 3.1} is the same as Theorem 1.1 in \cite{Leonetti-Petricca}.

\vspace{2mm}

Let us take a special case $g(t) =t \ln (e+t)$ as in (\ref{condition for g-2}). In this case, (\ref{nonstandard growth condition}) becomes
\begin{equation}\label{nonstandard growth condition-1}
\nu |z|^p \ln ^p (e+|z|)-a(x) \le f(x,z)
\end{equation}
for almost all $x\in \Omega$ and for all $z\in \mathbb R^n$. We have

\begin{corollary}\label{corollary-1}
Assume that $a(x) ,f(x,Du_*(x)) \in L^\sigma
(\Omega)$ where $\sigma \in (1,+\infty)$. If $u\in \cal C$
minimizes the variational integral (\ref{variational integral}) under (\ref{nonstandard growth condition-1}), then
\begin{equation*}
\begin{array}{llll}
&\displaystyle \sigma >\frac n p \Rightarrow |u-u_*| \ln (e+|u-u_*|) \in L^\infty (\Omega);\\
&\displaystyle  \sigma =\frac n p \Rightarrow  \exists \tilde \tau (n,p,\mu,\nu,\|f(x,Du_*)+a\|_{L^\sigma (\Omega)}) >0 \mbox { such that } e ^{\frac {|u-u_*|}{2\tilde \tau}}  \in L^1 (\Omega); \\
&\displaystyle \sigma <\frac n p \Rightarrow |u-u_*| \ln (e+|u-u_*|) \in L_w ^{\frac {np\sigma }{n-p\sigma }} (\Omega).
\end{array}
\end{equation*}
\end{corollary}

\vspace{2mm}

Assume now that we have $g^ p$ growth from below and $g^q$ growth from above: for some positive constants $p, q, \nu,K$
 with $1 < p < n$, $p \le q$, for some functions $a_1, a_2: \Omega \rightarrow [0,+\infty) $ we have
$$
\nu g^p(|z|) -a_1(x) \le f(x,z) \le K g^q(|z|) +a_2 (x)
$$
for almost all $x\in \Omega $ and all $z\in \mathbb R^n$, then
$$
\begin{array}{llll}
|f(x,z)| &\le \displaystyle K\left( g(|z|)+1\right) ^q  +a_1(x) +a_2(x)\\
& \le \displaystyle 2 ^{q-1}K g^q(|z|) +2 ^{q-1}K +a_1(x) +a_2(x) ,
\end{array}
$$
and the $\sigma$-integrability of $|f(x,Du_*(x)|$ is guaranteed by $g(|Du_*(x) |) \in L^{q\sigma} (\Omega)$ and $a_1,a_2\in L^\sigma $: this gives the following

\begin{corollary}\label{corollary}
Assume that $a_1(x),a_2(x) \in L^\sigma (\Omega)$ and $g(|Du_*(x) |)\in L^{q\sigma } (\Omega)$ where $\sigma \in (1,+\infty)$.
If $u\in \cal C$ minimizes the variational integral (\ref{variational integral}) under (\ref{nonstandard growth condition}), then the results in (\ref{3.5-new}) hold true.
\end{corollary}

\noindent{\it Proof of Theorem \ref{theorem 3.1}.}
Let $u\in \cal C$ be a minimizer of the variational integral (\ref{variational integral}) under (\ref{nonstandard growth condition}). For a constant $k >0$, let us define
$$
A_k  =\{x\in \Omega: |u(x)-u_*(x)| >k \}.
$$
We apply (\ref{condition for g-1}),  the convexity of $g$ and the $g^ p$ growth from below in (\ref{nonstandard growth condition}),  and we get
\begin{equation}\label{3.7}
\begin{array}{llll}
&\displaystyle \int_{A_k } g^p (|Du-Du_*|) dx\\
\le &\displaystyle  \int_{A_k } g^p \left(2 \frac {|Du|+|Du_*|} 2\right) dx \\
\le &\displaystyle 2^{\mu p} \int_{A_k } g^p \left(\frac {|Du|} 2 + \frac {|Du_*|}{2}\right) dx \\
\le &\displaystyle 2^{\mu p} \int_{A_k }\frac 1 {2^p} \left( g (|Du|)+ g(|Du_*|) \right) ^p dx \\
\le &\displaystyle 2^{\mu p-1} \int_{A_k}\left( g^p (|Du|) +g^p(|Du_*|)\right) dx \\
\le &\displaystyle \frac {2^{ \mu  p-1}} \nu \int_{A_k} \left(f(x,Du)+f(x,Du_*) +2a\right) dx.
\end{array}
\end{equation}
In order to control $\int_{A_k} f(x,Du)dx $ we need the minimality of $u$. Define the variation $v$ as follows
$$
v =G_k(u-u_*) =
\begin{cases}
u-u_*+k, & \mbox { if } u-u_*<-k, \\
0, & \mbox { if } |u-u_*| \le k, \\
u-u_*-k, &\mbox { if } u-u_* >k.
\end{cases}
$$
It turns out that
$$
Dv= (Du-Du_*)\cdot 1_{A_k},
$$
where $1_E(x)$ is the characteristic function over the set $E$, that is, $1_E(x) =1$ for $x\in E$ and $1_E(x)=0$ otherwise. Now we consider
$w =u-v$, and we have
$$
w \in u_* +W_0 ^{1,1} (\Omega) \ \ \mbox { and } \ \ Dw =Du \cdot 1 _{\Omega \setminus A_k} +Du_*\cdot  1 _{A_k}.
$$
Since $u$ and $u_*$ have finite energy, then the above inequality tells us that $w $ has finite energy too:
$$
f(x,Dw(x)) \in L^1(\Omega).
$$
Thus $w \in \cal C$ and we can use minimality (\ref{minimallity}):
$$
\begin{array}{llll}
&\displaystyle \int_{\Omega \setminus A_k} f(x,Du)dx +\int_{A_k} f(x,Du)dx =\int_{\Omega } f(x,Du)dx\\
\le &\displaystyle \int_{\Omega } f(x,Dw)dx = \int_{\Omega \setminus A_k} f(x,Du)dx +\int_{A_k} f(x,Du_*)dx.
\end{array}
$$
Since $u$ and $u_*$ have finite energy, all the integrals are finite; then we can drop $\int_{\Omega \setminus A_k} f(x,Du)dx$ from both sides and we get
$$
\int_{A_k} f(x,Du)dx \le \int_{A_k} f(x,Du_*)dx.
$$
This inequality can be used in (\ref{3.7}) and we get
\begin{equation}\label{3.8}
\begin{array}{llll}
&& \displaystyle
\int_{A_k} g^p (|Du-Du_*|) dx \\
&\le &\displaystyle  \frac {2^{ \mu p}}\nu  \int_{A_k}  ( f(x,Du_*) +a) dx\\
&\le &\displaystyle \frac {2^{ \mu p}}\nu \|f(x,Du_*)+a\|_{L^\sigma (\Omega)} |A_k| ^{1-\frac 1 \sigma},
\end{array}
\end{equation}
where we used H\"older inequality (\ref{Holder}). Since $u, u_*\in \cal C$, then $f(x,Du), f(x,Du_*) \in L^1(\Omega)$, (\ref{nonstandard growth condition}) tells us that
$g(|Du|), g(|Du_*|)$ $ \in L^p(\Omega)$. We use Lemma \ref{lemma 2} and we have $g(|u|), g(|u_*|) \in L^{p*} (\Omega)$.
Then (\ref{condition for g-1}) together with the convexity property of $g$ implies
\begin{equation}\label{No.1}
\begin{array}{llll}
\displaystyle g(|u-u_*|) &\le &\displaystyle  g \left(2 \frac {|u|+|u_*|}{2}\right)  \le 2 ^\mu g \left (\frac {|u|}{2} +\frac {|u_*|}{2}\right)\\[3mm]
& \le &\displaystyle 2 ^{\mu -1} \left(g(|u|) +g(|u_*|)\right) \in L^{p*} (\Omega).
\end{array}
\end{equation}
Similarly,
\begin{equation}\label{No.2-1}
\begin{array}{llll}
\displaystyle g(|D(u-u_*)|) &\le &\displaystyle  g \left(2 \frac {|Du|+|Du_*|}{2}\right) \le 2 ^\mu g \left (\frac {|Du|}{2} +\frac {|Du_*|}{2}\right)\\[3mm]
& \le &\displaystyle 2 ^{\mu -1} \left(g(|Du|) +g(|Du_*|)\right) \in L^{p} (\Omega).
\end{array}
\end{equation}
(\ref{No.1}), (\ref{No.2-1}) combined with (\ref{new inequality}) and (\ref{condition for g-3}) imply
\begin{equation}\label{No.2}
\begin{array}{llll}
&\displaystyle |Dg((|u-u_*|-k)_+)|\\[3mm]
=&\displaystyle  g'((|u-u_*|-k)_+) |D((|u-u_*|-k)_+)|  \\[3mm]
\le &\displaystyle g'((|u-u_*|-k)_+) (|u-u_*|-k)_+\\[3mm]
  &\displaystyle + g'(|D((|u-u_*|-k)_+)| ) |D((|u-u_*|-k)_+)| \\[3mm]
\le &\displaystyle \mu \left[ g((|u-u_*|-k)_+) + g(|D((|u-u_*|-k)_+)| )\right] \in L^p(\Omega).
\end{array}
\end{equation}
(\ref{No.1}) and (\ref{No.2}) tell us that
$$
g\left( (|u-u_*|-k )_+\right) \in W_0^{1,p} (\Omega),
$$
which allows us to use Sobolev inequality
\begin{equation}\label{Sobolev inequality}
u\in W_0^{1,p} (\Omega) \Longrightarrow \|u\| _{L^{p^*}(\Omega)} \le {\cal S} \|Du\|_{L^p(\Omega)}, \ \ p\ge 1, \ {\cal S} ={\cal S} (n,p),
\end{equation}
and we have, for $h> k>0$,
\begin{equation}\label{3.12}
\begin{aligned}
 & \int _{A_{k} }| Dg ((| u-u_*|-k)_+)|^pdx \\
=&\int _{\Omega }  | Dg (  ( |  u-u_{\ast }| -k   )_{+  }    )   |^{p}dx \\
\ge & {\cal S}^{-p} \left( \int_{A_{k}}  ( g (  ( |  u -u_{\ast }| -k   )_{+  }    ))^{p^{\ast } }  dx \right )^{\frac{p}{p^{\ast } } } \\
\ge& {\cal S}^{-p}  \left( \int _{A_{h} }  ( g (  ( |  u -u_{\ast } | -k   )_{+  }    ))^{p^{\ast } } dx  \right )^{\frac{p}{p^{\ast } } }\\
\ge&{\cal S}^{-p}  \left( \int _{A_{h} } ( g ( h -k   ))^{p^{\ast } }  dx  \right)^{\frac{p}{p*} }\\
= & {\cal S}^{-p}  g^{p}  ( h-k  )  | A_{h}   |^{\frac{p}{p^{\ast } } }.
\end{aligned}
\end{equation}
In order to estimate the left hand side of (\ref{3.12}), we use again Sobolev inequality, (\ref{No.2}) and (\ref{3.8}), then
\begin{equation}\label{3.13}
\begin{aligned}
&\int _{A_{k} } |Dg((|u-u_*|-k)_+)|^pdx \\
\le & 2^{p-1}\mu^p \left[ \int_{A_{k}}  g^p((|u-u_*|-k)_+) dx +  \int_{A_{k}} g^p (|D((|u-u_*|-k)_+)| ) dx \right] \\
=& 2^{p-1}\mu^p \left[ \int_{\Omega}  g^p((|u-u_*|-k)_+) dx +  \int_{A_k} g^p (|D((|u-u_*|-k)_+)| ) dx \right].
\end{aligned}
\end{equation}
The first term in the right hand side of the above inequality can be estimated by using H\"older inequality,
\begin{equation}\label{3.14}
\begin{aligned}
& \int_{\Omega}  g^p((|u-u_*|-k)_+) dx \\
\le & {{\cal C}_*} \left(\int_\Omega | Dg((|u-u_*|-k)_+)|^{p_*}dx \right)^{\frac {p}{p_*}} \\
= & {{\cal C}_*} \left(\int_{A_k} | Dg((|u-u_*|-k)_+)|^{p_*}dx \right)^{\frac {p}{p_*}} \\
\le & {{\cal C}_*} \left[ \left(\int_{A_k} | Dg((|u-u_*|-k)_+)|^{p}dx \right )^ {\frac {p_*}{p}} \left(\int_{A_k} dx \right)
^{\frac {p-p_*}{p}} \right]^{\frac {p}{p_*}} \\
=& {{\cal C}_*} \int _{A_{k} }|Dg((|u-u_*|-k)_+) |^p dx \cdot |A_k| ^{\frac {p-p_*}{p_*}},
\end{aligned}
\end{equation}
where $p_*=\frac {np}{n+p}$, ${{\cal C}_*} $ is a constant depending only on $n,p$. There exists a constant $k_0>0$ such that for all $k\ge k_0$,
\begin{equation}\label{choice of k0}
2^{p-1}\mu^p{{\cal C}_*} |A_k| ^{\frac {p-p_*}{p_*}} <\frac 1 2.
\end{equation}
For this choice of $k_0$ and (\ref{3.14}), the first term in the right hand side of (\ref{3.13})
 can be absorbed by the left hand side, thus (\ref{3.13}) implies
\begin{equation}\label{3.15}
\int _{A_{k} } |Dg((|u-u_*|-k)_+)|^pdx \le c\int_{A_k} g^p (|Du-Du_*| ) dx.
\end{equation}
Combining (\ref{3.12}), (\ref{3.15}) and (\ref{3.8}), we arrive at
\begin{equation*}
|A_h| \le \frac {c}{g^{p*}(h-k)} |A_k| ^{\left(1-\frac 1 \sigma\right) \frac {p*}{p}}, \ \ \forall k\ge k_0,
\end{equation*}
where $c$ is a constant depending only on $n,p,\mu,\nu$ and $\|f(x,Du_*)+a_1\|_{L^\sigma (\Omega)}$, and $k_0$ satisfies (\ref{choice of k0}).
Therefore, (\ref{stampacchia_assumption generalized-generalization-new}) holds true for
$$
\varphi(k)=|A_k|, \ \alpha =p^* \ \mbox { and } \ \beta =\left(1-\frac 1 \sigma\right) \frac {p^*}{p}.
$$
We use Corollary \ref{Stampacchia Lemma, Gao Zhang and Ma generalized,generalized-new} and we have

\vspace{2mm}

{\bf Case 1}: $\sigma >\frac{n}{p}$. In this case $\beta >1$ and we use Corollary \ref{Stampacchia Lemma, Gao Zhang and Ma generalized,generalized-new}-(i):
there exists a constant $L$ such that
$$
|A_{2L}|=|\{|u-u_*|>2L\}| =0,
$$
thus $|u-u_*|\le 2L$, a.e. $\Omega$.

\vspace{2mm}

{\bf Case 2}: $\sigma =\frac{n}{p}$. In this case $\beta =1$ and we use Corollary \ref{Stampacchia Lemma, Gao Zhang and Ma generalized,generalized-new}-(ii):  for any $k\ge k _0$,
\begin{equation} \label{statement_stampacchia_beta=1 new new}
|\{ |u-u_*|>k\}| =\varphi (k) \le \varphi (k_0) e ^{1-\frac {k-k_0}{\tilde \tau}} =\varphi (k_0) e^{1+\frac {k_0}{\tilde \tau}} e ^{-\frac {k}{\tilde \tau}}:=\bar C e ^{-\frac {k}{\tilde \tau}},
\end{equation}
where $\tilde \tau$ is the constant satisfying (\ref{c2.3 new new}) and (\ref{tau,generalized-new new}). (\ref{statement_stampacchia_beta=1 new new}) implies
\begin{equation*}\label{12.13-1}
\left|\left\{ e ^{\frac {|u-u_*|}{2 \tilde \tau}} > e ^{\frac k {2\tilde \tau}}\right\}\right| =|\{|u-u_*|>k\}| \le \bar C e ^{-\frac {k}{\tilde \tau}}.
\end{equation*}
Denote $\tilde k =e ^{\frac k {2\tilde \tau}}$, then the above inequality is equivalent to
\begin{equation*}\label{12.13-1}
\left|\left\{ e ^{\frac {|u-u_*|}{2 \tilde \tau}} > \tilde k \right\}\right| \le \frac {\bar C} {\tilde k ^2}, \ \ \  \forall \tilde k \ge \tilde k_0 = e ^{\frac {k_0}{2 \tilde \tau}}.
\end{equation*}
%
%
We use Lemma 3.11 in \cite{Boccardo-Croce}, which states that a necessary and sufficient condition for $g\in L^r (\Omega)$, $r\ge 1$, is
$$
\sum_{k=1} ^\infty k^{r-1} |\{ |g|>k\}| <\infty.
$$
We use the above lemma with $g=e ^{\frac {|u-u_*|}{2 \tilde \tau}} $ and $r=1$, since
$$
\sum_{\tilde k = [\tilde k _0]+1 } ^\infty \left| \left\{e ^{\frac {|u-u_*|}{2 \tilde \tau}}  >\tilde k \right\} \right|
 \le \sum_{\tilde k =[\tilde k _0]+1} ^\infty \frac {\bar C}{\tilde k ^2} =\hat C <\infty,
$$
then
$$
\sum_{\tilde k =1} ^\infty \left| \left\{e ^{\frac {|u-u_*|}{2 \tilde \tau}}  >\tilde k \right\} \right|
= \left( \sum_{\tilde k =1} ^{[\tilde k_0]} +\sum_{\tilde k =[\tilde k_0]+1 } ^\infty \right)
\left| \left\{e ^{\frac {|u-u_*|}{2 \tilde \tau}}  >\tilde k \right\} \right| \le [\tilde k_0]|\Omega|+\hat C  <\infty,
$$
that is,
$$
e ^{\frac {|u-u_*|}{2 \tilde \tau}}  \in L^1(\Omega).
$$


%

\vspace{2mm}

{\bf Case 3}: $\sigma <\frac{n}{p}$. In this case $0< \beta <1$ and we use Corollary \ref{Stampacchia Lemma, Gao Zhang and Ma generalized,generalized-new}-(iii):
\begin{equation}\label{the last one-12}
|\{ |u-u_*|>k\}| \le c \left(\frac 1 {g(k)}\right) ^{\frac {np\sigma}{n-p\sigma}}, \ \ \forall k\ge k_0.
\end{equation}
For $0<k<k_0$, one has
\begin{equation}\label{the last one-2}
|\{ |u-u_*|>k\}| \le |\Omega| \le |\Omega| g(k_0)^{\frac {np\sigma}{n-p\sigma}} \left(\frac 1 {g(k)}\right)^{\frac {np\sigma}{n-p\sigma}}.
\end{equation}
(\ref{the last one-12}) and (\ref{the last one-2}) imply, for all $k>0$,
$$
|\{ |u-u_*|>k\}| \le \max \left\{c, |\Omega| g(k_0)^{\frac {np\sigma}{n-p\sigma}} \right\} \left(\frac 1 {g(k)}\right)^{\frac {np\sigma}{n-p\sigma}}.
$$
The above inequality is equivalent to
\begin{equation*}\label{the last one-1}
|\{g( |u-u_*|)>k\}| \le \max \left\{c, |\Omega| g(k_0)^{\frac {np\sigma}{n-p\sigma}} \right\} \left(\frac 1 {k}\right) ^{\frac {np\sigma}{n-p\sigma}}, \ \ \forall k>0,
\end{equation*}
then the desired result $g(|u-u_*|) \in L_w ^{\frac {np\sigma}{n-p\sigma}} (\Omega)$ follows, completing the proof of Theorem \ref{theorem 3.1}.
\qed

\section{The second generalization and applications.}
\noindent In this section we give the second generalization of Lemma \ref{Stampacchia Lemma, Gao Zhang and Ma generalized}, that is, (\ref{stampacchia_assumption-2}) is replaced by (\ref{the second condition}) with the function $g$ the one introduced in Section 1, and give an application to degenerate elliptic equations.
We prove the following

\begin{lemma}\label{Stampacchia Lemma, generalization}  Let $g:[0,+\infty)\rightarrow [0,+\infty)$ be a function satisfying the assumptions (${\cal G}_1$), (${\cal G}_2$) and (${\cal G}_3$); let $  c, \alpha, \beta,k_0$ be positive constants and $\theta \ge 0$; let
$\varphi: [k_0,+\infty) \rightarrow [0, +\infty)$ be nonincreasing and such that
\begin{equation}\label{stampacchia_assumption 22}
\varphi (h) \le \frac {c g ^{\theta \alpha } (h)}{(h-k)^\alpha}
[\varphi (k)]^\beta
\end{equation}
\noindent
for every $h,k$ with $h>k\ge k_0 >0$.
It results that:

{\bf (i)} if $\beta > 1$ and $\theta < 1$
\begin{equation}\label{19.123}
\lim_{L\rightarrow +\infty} \frac {g^\theta (L)}{L} =0,
\end{equation}
then
$$
\varphi (2L) =0,
 $$
where $L\ge 2k_0$ satisfies
\begin{equation}\label{the value of L-2}
  \frac{g^{\theta }(L) }{L} \le c^{-\frac{1}{\alpha } }2^{-\frac{1}{\beta }(\mu \theta
  	+  \beta +\frac{1}{\beta -1} ) }  [\varphi (k_{0} )]^{\frac{1-\beta }{\alpha } } .
\end{equation}

{\bf(ii)} if $\beta = 1$, $\mu \theta <1$,  and there exists $\tilde \theta >\theta $ such that
\begin{equation}\label{19.2}
\lim _{L\rightarrow +\infty} \frac {g^{\tilde \theta } (L)}{L} =0,
\end{equation}
then for any $k\ge k_0$,
\begin{equation*} \label{generalied statement_stampacchia_beta=1}
\varphi (k) \le \varphi(k_0) e ^{1-\left(\frac {k-k_0}{\tau}\right) ^{1-\frac {\theta }{\tilde \theta}}} ,
\end{equation*}
where $\tau \ge \max \left \{ k_{0} ,\frac{1}{2}  \right \}$ is large enough such that
\begin{equation}\label{tau111}
\frac {c g^{\theta \alpha } (k_0+\tau)}{\tau ^\alpha} \le \frac 1 e
\end{equation}
and
\begin{equation}\label{tau11}
 \frac {c\left(2 ^{\frac {\tilde \theta }{\tilde \theta -\theta}+1 } \right)^{\mu \theta \alpha}\tau ^{(\mu \theta -1)\alpha}  }{\left(\frac { \tilde \theta }{\tilde \theta -\theta}\right) ^\alpha } \left(\max _{s\ge 1}\frac {g^\theta \left(s ^{\frac {\tilde \theta }{\tilde \theta -\theta }}\right)}{s^{\frac {\theta }{\tilde \theta -\theta }}} \right)^\alpha <\frac 1 e;
\end{equation}

{\bf(iii)} if $0< \beta < 1$, $0\le \theta <1$, and there exists $0< \epsilon _0 <1-\theta$ such that
\begin{equation}\label{tau11111}
\lim _{k\rightarrow +\infty} \frac {g'(k)}{g ^{1-\theta -\epsilon _0}(k)} =0,
\end{equation}
then for any $k\ge k_0$,
\begin{equation*}\label{statement_stampacchia_beta<1}
\varphi (k) \le  2 ^{\frac {\mu \epsilon_0\alpha(2-\beta) }{(1-\beta )^2}} \left(T ^{\frac 1 {1-\beta}}+\varphi (k_0)g^{\frac {\epsilon_0\alpha}{1-\beta}}(k_0) \right) \left(\frac 1 {g(k)}\right) ^{\frac {\epsilon_0 \alpha}{1-\beta}} ,
\end{equation*}
where
$$
T= c 2 ^{\mu\theta \alpha} \left[\max _{k\ge k_0} \left(\frac {g'(k)}{g ^{1-\theta-\epsilon _0}(k)}\right) \right]^\alpha.
$$
\end{lemma}



\begin{remark}
We remark that (\ref{the value of L-2}) holds true for sufficiently large $L$ because of (\ref{19.123}).
\end{remark}

\noindent {\bf Proof of Lemma \ref{Stampacchia Lemma, generalization}}.

(i) For $\beta >1$ we fix $L>k_0$ and choose levels
$$
k_i = 2L (1-2^{-i-1}), \ \  i=0,1,2,\cdots.
$$
It is obvious that $k_0<L\le k_i <2L$ and $\{ k_i \}$ be an increasing sequence. We choose in (\ref{stampacchia_assumption 22})
$$
k=k_i,\  h=k_{i+1}, \ x_i = \varphi (k_i) \  \mbox { and  }  \ x_{i+1} = \varphi (k_{i+1}),
$$
and noticing that
$$
h-k = k_{i+1}-k_i=L 2 ^{-i-1},
$$
we have, for $i=0,1,2,\cdots, $
$$
\begin{aligned}
x_{i+1} & \le \frac {c g^{\theta \alpha}( 2L (1-2^{-i-2})) }{ (L2^{-i-1}) ^\alpha} x_i^\beta
\le \frac {c g^{\theta \alpha}( 2L)}{ (L2^{-i-1}) ^\alpha} x_i^\beta \\
& \le \frac {c 2 ^{\mu \theta \alpha }g^{\theta \alpha}(L)}{ (L2^{-i-1}) ^\alpha} x_i^\beta
= c 2 ^{(\mu \theta +1) \alpha} \left(\frac {g^\theta (L)}{L}\right) ^\alpha 2 ^{i\alpha} x_i^\beta ,
\end{aligned}
$$
where we used again (\ref{condition for g-1}). 
Thus (\ref{generalized Stampachia lemma proof need}) holds true with 
$$
C=c 2 ^{(\mu \theta +1) \alpha} \left(\frac {g^\theta (L)}{L}\right) ^\alpha \  \ \mbox { and } \  \ B=2^\alpha.
$$
We get from Lemma \ref{Stampacchia Lemma proof need}
\begin{equation}\label{limit 1}
\lim _{i\rightarrow +\infty} x_i =0
\end{equation}
provided that
\begin{equation}\label{condition 1}
x_0 =\varphi (k_0) =\varphi (L) \le \left(c 2 ^{(\mu \theta +1) \alpha} \left(\frac {g^\theta (L)}{L}\right) ^\alpha\right) ^{-\frac 1 {\beta -1}} \left(2 ^\alpha \right)  ^{-\frac 1 {(\beta -1)^2 }} .
\end{equation}
Note that (\ref{limit 1}) implies
$$
\varphi (2L) =0.
$$
Let us check condition (\ref{condition 1}) and determine the value of $L$. (\ref{condition 1}) is equivalent to
\begin{equation}\label{limit 2}
\varphi (L) \le c^{-\frac 1 {\beta -1}} 2 ^ {-\frac {\alpha}{\beta-1} (\mu \theta +1+\frac 1 {\beta -1}) } \left(\frac {g^\theta (L)}{L}\right) ^{-\frac \alpha {\beta -1}}.
\end{equation}
In (\ref{stampacchia_assumption 22}) we take $k=k_0$ and $h=L\ge 2k_0$ (which is equivalent to $L-k_0 \ge \frac L 2$) and we have
$$
\varphi (L) \le \frac {cg^{\theta \alpha } (L)}{(L-k_0)^\alpha } [\varphi (k_0)] ^\beta \le c 2 ^\alpha \left( \frac { g^{\theta}(L)}{L} \right)^\alpha [\varphi (k_0)] ^\beta.
$$
Then condition (\ref{limit 2}) would be satisfied if $L\ge 2k_0 $ and
$$
c 2 ^\alpha \left( \frac { g^{\theta}(L)}{L} \right)^\alpha [\varphi (k_0)] ^\beta \le c ^{-\frac 1 {\beta -1}} 2 ^ {-\frac {\alpha}{\beta-1} (\mu \theta +1+\frac 1 {\beta -1}) } \left(\frac {g^\theta (L)}{L}\right) ^{-\frac \alpha {\beta -1}}.
$$
The above inequality is equivalent to (\ref{the value of L-2}).

(ii)  Let  $\beta =1$ and $\tau$ satisfies (\ref{tau111}) and (\ref{tau11}). For $s=0,1,2,\cdots, $ we let
$$
k_s =k_0+ \tau  s^{\frac {\tilde \theta} {\tilde \theta-\theta }},
$$
then $\{k_s\}$ is an increasing sequence and
$$
k_{s+1}-k_s =\tau \left[ (s+1) ^{\frac {\tilde \theta} {\tilde \theta-\theta }} -s^{\frac {\tilde \theta} {\tilde \theta-\theta }}\right].
$$
We use Taylor's formula in order to get
\begin{equation}\label{estimate 10}
k_{s+1}-k_s =\tau \left[\frac {\tilde \theta }{\tilde \theta -\theta}s ^{\frac {\theta }{\tilde \theta -\theta}} +\frac {\tilde \theta \theta}{2!(\tilde \theta -\theta)^2} \xi ^{\frac {2 \theta -\tilde \theta }{\tilde \theta -\theta}}
\right]\ge \frac {\tau \tilde \theta }{\tilde \theta -\theta}s ^{\frac {\theta }{\tilde \theta -\theta}} ,
\end{equation}
where $\xi$ lies in the open interval $(s,s+1)$. In (\ref{stampacchia_assumption 22}) we take $\beta =1$, $k=k_s$ and $h=k_{s+1}$, we use (\ref{estimate 10}) and we get, for $s\ge 1$,
\begin{equation}\label{estimate 11}
\begin{aligned}
\varphi (k_{s+1}) &\le \frac {c g ^{ \theta  \alpha}  \left(k_0 +\tau (s+1)^{\frac {\tilde \theta} {\tilde \theta- \theta }}\right)}{\left( \frac {\tau \tilde \theta }{\tilde \theta -\theta} \right)^\alpha s ^{\frac {\theta \alpha}{\tilde \theta -\theta}}} \varphi (k_s) \\
 &\le \frac {c g ^{ \theta  \alpha}  \left(2\tau (2s)^{\frac {\tilde \theta} {\tilde \theta- \theta }}\right)}{\left( \frac {\tau \tilde \theta }{\tilde \theta -\theta}\right)^\alpha s ^{\frac {\theta \alpha}{\tilde \theta -\theta}}} \varphi (k_s) \\
& \le \frac {c \left(2 ^{\frac {\tilde \theta }{\tilde \theta -\theta}+1 } \tau \right)^{\mu \theta \alpha} g^{\theta \alpha }\left(s^ \frac {\tilde \theta }{\tilde \theta -\theta }\right) }{\left(\frac {\tau \tilde \theta }{\tilde \theta -\theta}\right) ^\alpha  s ^{\frac {\theta \alpha }{\tilde \theta -\theta }}} \varphi(k_s).
\end{aligned}
\end{equation}
We next show that, for any $s \ge 1$,
$$
\frac {g^\theta \left(s ^{\frac {\tilde \theta }{\tilde \theta -\theta }}\right)}{s^{\frac {\theta }{\tilde \theta -\theta }}}
$$
is finite. In fact, by (\ref{19.2}),
$$
\lim_{s\rightarrow +\infty} \frac {g^\theta \left(s ^{\frac {\tilde \theta }{\tilde \theta -\theta }}\right)}{s^{\frac {\theta }{\tilde \theta -\theta }}}
=\lim_{\tilde s\rightarrow +\infty} \frac {g^\theta (\tilde s )}{\tilde s ^{\frac {\theta}{\tilde \theta}}} =
\left( \lim_{\tilde s\rightarrow +\infty} \frac {g^{\tilde \theta} (\tilde s )}{\tilde s }  \right)^{\frac {\theta }{\tilde \theta}} =0.
$$
(\ref{estimate 11}) and (\ref{tau11}) ensure
$$
\varphi (k_{s+1}) \le \frac {c \left(2 ^{\frac {\tilde \theta }{\tilde \theta -\theta}+1 } \right)^{\mu \theta \alpha}\tau ^{(\mu \theta -1)\alpha}  }{\left(\frac { \tilde \theta }{\tilde \theta -\theta}\right) ^\alpha } \left(\max _{s\ge 1}\frac {g^\theta \left(s ^{\frac {\tilde \theta }{\tilde \theta -\theta }}\right)}{s^{\frac {\theta }{\tilde \theta -\theta }}} \right)^\alpha  \varphi(k_s)\le \frac 1 e \varphi (k_s).
$$
By recursion,
\begin{equation}\label{tau1111}
\varphi (k_s) \le \frac 1 {e^s } \varphi (k_0),\ \ s=1,2,\cdots.
\end{equation}
In (\ref{stampacchia_assumption 22}) we take $k=k_0$ and $h=k_0+\tau$, (\ref{tau111}) ensures (\ref{tau1111}) holds true also for $s=0$.

For any $k\ge k_0$, there exists $s\in \{1,2,\cdots\}$ such that
$$
k_0 +\tau (s-1) ^{\frac {\tilde \theta } {\tilde \theta- \theta }}  \le k <  k_0 +\tau s ^{\frac {\tilde \theta } {\tilde \theta- \theta }} .
$$
Thus, considering $\varphi (k)$ is nonincreasing, using (\ref{tau1111}), one obtains
$$
\varphi (k) \le \varphi \left( k_0 +\tau (s-1) ^{\frac {\tilde \theta } {\tilde \theta- \theta }} \right)=\varphi (k_{s-1})
\le e^{1-s} \varphi (k_0)\le \varphi(k_0) e ^{1-\left(\frac {k-k_0}{\tau}\right) ^{1-\frac {\theta }{\tilde \theta}}}.
$$

(iii) For the case $0<\beta <1$, let us take $h=2k$ in (\ref{stampacchia_assumption 22}) and we get for every $k\ge k_0> 0$,
\begin{equation}\label{tau111111}
\varphi(2k) \le \frac {cg^{\theta \alpha} (2k)}{k^\alpha} [\varphi (k)] ^\beta \le \frac {c2 ^{\mu \theta \alpha}g^{\theta \alpha} (k)}{k^\alpha} [\varphi (k)] ^\beta .
\end{equation}
Since
$$
g(k) =g(k) -g(0) =g'(\xi) k \le g'(k)k,
$$
where $\xi \in (0,k)$, then (\ref{tau111111}) yields
$$
\begin{aligned}
\varphi(2k) & \le \frac {c2 ^{\mu \theta \alpha} (g'(k))^{\alpha}}{g^{(1-\theta )\alpha}(k)} [\varphi (k)] ^\beta \\
& =c 2 ^{\mu \theta \alpha} \left(\frac {g'(k)}{g^{1-\theta -\epsilon _0} (k)} \right)^\alpha \frac {1}{g^{\epsilon_0 \alpha} (k)} [\varphi (k)]^\beta \\
& \le c 2 ^{\mu \theta \alpha} \left[\max _{k\ge k_0}\frac {g'(k) }{g^{1-\theta -\epsilon _0} (k)} \right]^\alpha \frac {1}{g^{\epsilon_0 \alpha} (k)} [\varphi (k)]^\beta\\
& =\frac {T}{g^{\epsilon_0\alpha }(k)} [\varphi (k)] ^\beta,
\end{aligned}
$$
where we have used (\ref{tau11111}), which ensures the term in the square bracket is finite. Let us introduce a new function $\psi (k): [k_0,+\infty) \rightarrow [0,+\infty]$ such that
$$
\varphi (k ) =\psi(k) \frac {T^{\frac 1 {1-\beta}}}{g^{\frac {\epsilon_0 \alpha }{1-\beta}} (k)},
$$
then we get
\begin{equation}\label{19-3}
\psi (2k)\le \left(\frac {g(2k)}{g(k)}\right)^{\frac {\epsilon_0 \alpha}{1-\beta}} \psi ^\beta (k)\le 2 ^{\frac {\mu \epsilon _0 \alpha}{1-\beta}} \psi ^\beta (k), \ \ k\ge k_0.
\end{equation}
For any $k\ge k_0$, one can find a natural number $s\ge 1 $ such that
$$
2^{s-1} k_0 \le k < 2 ^s k_0,
$$
for such a $k$ one has
\begin{equation}\label{2.9-9}
\varphi (k)\le \varphi (2^{s-1}k_0)=\psi (2^{s-1}k_0) \frac {T^{\frac 1 {1-\beta}}}{g^{\frac {\epsilon_0\alpha}{1-\beta}}(2^{s-1} k_0) }.
\end{equation}
(\ref{19-3}) implies, for $s\ge 1$,
\begin{equation}\label{2.10-9}
\begin{array}{llll}
 \displaystyle \psi (2^{s-1}k_0) &\le & \displaystyle  2 ^{\frac {\mu \epsilon_0\alpha}{1-\beta } \sum \limits_{i=0}^{s-2}\beta ^i}  \psi ^{\beta ^{s-1}} (k_0)
\le 2 ^{\frac {\mu \epsilon_0\alpha}{(1-\beta )^2}}(1+\psi (k_0)) \\
 & \le & \displaystyle  2 ^{\frac {\mu \epsilon_0\alpha }{(1-\beta )^2}} \left(1+\varphi (k_0)g^{\frac {\epsilon_0\alpha}{1-\beta}}(k_0) T^{\frac 1 {\beta-1}}  \right).
\end{array}
\end{equation}
Substituting (\ref{2.10-9}) into (\ref{2.9-9}) we arrive at
\begin{equation*}\label{2.12}
\begin{aligned}
\varphi (k) & \le  2 ^{\frac {\mu \epsilon_0\alpha }{(1-\beta )^2}} \left(1+\varphi (k_0)g^{\frac {\epsilon_0\alpha}{1-\beta}}(k_0) T^{\frac 1 {\beta-1}}  \right)\frac {T^{\frac 1 {1-\beta}}}{g^{\frac {\epsilon_0\alpha}{1-\beta}}(2^{s-1} k_0) } \\
&= 2 ^{\frac {\mu \epsilon_0\alpha }{(1-\beta )^2}} \left(T ^{\frac 1 {1-\beta}}+\varphi (k_0)g^{\frac {\epsilon_0\alpha}{1-\beta}}(k_0) \right) \left(\frac 1 {g(2 ^{s-1} k_0)}\right) ^{\frac {\epsilon_0 \alpha}{1-\beta}}\\
&\le 2 ^{\frac {\mu \epsilon_0\alpha (2-\beta) }{(1-\beta )^2}} \left(T ^{\frac 1 {1-\beta}}+\varphi (k_0)g^{\frac {\epsilon_0\alpha}{1-\beta}}(k_0) \right) \left(\frac 1 {g(k)}\right) ^{\frac {\epsilon_0 \alpha}{1-\beta}} .
\end{aligned}
\end{equation*}
This completes the proof of Lemma \ref{Stampacchia Lemma, generalization}.

\vspace{4mm}

We now consider a special case: $g(t) =t\ln (e+t)$. We have the following corollary of Lemma \ref{Stampacchia Lemma, generalization}.

\begin{corollary}\label{Stampacchia Lemma, generalization-2}  
Let $  c, \alpha, \beta,k_0$ be positive constants and $\theta \ge 0$; let
$\varphi: [k_0,+\infty) \rightarrow [0, +\infty)$ be nonincreasing
and such that
\begin{equation}\label{stampacchia_assumption 222}
\varphi (h) \le \frac {c h^{\theta \alpha } \ln ^{\theta \alpha} (e+h)}{(h-k)^\alpha}
[\varphi (k)]^\beta
\end{equation}
\noindent
for every $h,k$ with $h>k\ge k_0 >0$. It results that:

{\bf (i)} if $\beta > 1$ and $0\le \theta <1$, then
$$
\varphi (2L) =0,
 $$
where $L\ge 2k_0$ satisfies
\begin{equation*}\label{the value of L}
\frac {\ln ^\theta (e+L)}{L^{1-\theta}}  \le c^{-\frac{1}{\alpha } }2^{-\frac{1}{\beta }(\mu \theta
  	+  \beta +\frac{1}{\beta -1} ) }  [\varphi (k_{0} )]^{\frac{1-\beta }{\alpha}}.
\end{equation*}

{\bf(ii)} if $\beta = 1$, $\mu \theta <1$, then for any $\theta <\tilde \theta <1$ and any $k\ge k_0$,
\begin{equation*} \label{generalied statement_stampacchia_beta=1}
\varphi (k) \le \varphi(k_0) e ^{1-\left(\frac {k-k_0}{\tau}\right) ^{1-\frac {\theta }{\tilde \theta}}} ,
\end{equation*}
where $\tau \ge \max \left \{ k_{0} ,\frac{1}{2}  \right \}$ is large enough satisfying
\begin{equation*}\label{tau111-2}
\frac {c(k_0+\tau)^{\theta \alpha} \ln ^{\theta\alpha} (e+\tau+ k_0)}{\tau ^\alpha} \le \frac 1 e
\end{equation*}
and
\begin{equation*}\label{tau11-2}
 \frac {c \left(2 ^{\frac {\tilde \theta }{\tilde \theta -\theta}+1 } \right)^{\mu \theta \alpha}\tau ^{(\mu \theta -1)\alpha}  }{\left(\frac { \tilde \theta }{\tilde \theta -\theta}\right) ^\alpha } \left(\max _{s\ge 1} \frac {\ln ^\theta \left( e+s ^{\frac {\tilde \theta }{\tilde \theta -\theta}}\right)}{s ^{\frac {\theta (1-\tilde \theta )}{\tilde \theta -\theta }}}\right)^\alpha <\frac 1 e;
\end{equation*}

{\bf(iii)} if $0< \beta < 1$, $0\le \theta <1$, then for any $0<\epsilon_0 <1-\theta$ and  for any $k\ge k_0$,
\begin{equation*} \label{statement_stampacchia_beta<1}
\varphi (k) \le 2 ^{\frac {\mu \epsilon_0 \alpha (2-\beta)}{1-\beta}} \left(T^{\frac 1 {1-\beta}} +\varphi (k_0)k_0^{\frac {\epsilon_0\alpha}{1-\beta}} \ln^{\frac {\epsilon_0\alpha}{1-\beta}} (e+k_0)\right) \left(\frac 1 {g(k)} \right) ^{\frac {\epsilon_0\alpha}{1-\beta}},
\end{equation*}
where
$$
T= c2 ^{\mu\theta \alpha} \left[\max _{k\ge k_0} \frac{\ln(e+k)+\frac{k}{e+k} }{k^{1-\theta -\epsilon_0}\ln^{1-\theta -\epsilon_0} (e+k)}\right]^\alpha.
$$
\end{corollary}

\begin{remark}
In Corollary \ref{Stampacchia Lemma, generalization-2}-(iii), we used the fact that, for $g(t)=t\ln (e+t)$, (\ref{tau11111}) holds true for any $0<\epsilon_0 <1-\theta$.
\end{remark}

\vspace{4mm}

As an application of Corollary \ref{Stampacchia Lemma, generalization-2} we are interested in the following degenerate elliptic problem
\begin{equation}\label{BVP}
\begin{cases}
-\mbox{div} (a(x,u(x)) Du (x)) =f(x), & \mbox { in } \Omega, \\
u(x)=0,  & \mbox { on } \partial \Omega,
\end{cases}
\end{equation}
here $\Omega$ is a bounded, open subset of $\mathbb R^n$, $n>2$, and $a(x,s):\Omega \times \mathbb R \rightarrow \mathbb R$ is a Carath\'eodory function (that is, measurable with respect to $x$ for every $s\in \mathbb R$
and continuous with respect to $s$ for almost every $x\in \Omega$) satisfying the following conditions
\begin{equation}\label{conditon for a(x,s)}
\frac {\alpha}{(1+|s|)^\theta \ln ^\theta (e+|s|)} \le a (x,s) \le \beta,
\end{equation}
for some positive number $\theta \ge 0$, for almost every $x\in \Omega$ and every $s\in \mathbb R$, where $\alpha$ and $\beta$ are positive constants. As far as the source $f$ is concerned, we assume
\begin{equation*}\label{condition for f}
f\in L_{w} ^m(\Omega), \  \ m>(2^*)'  =\frac {2n}{n+2}.
\end{equation*}

\begin{definition}
A function $u\in W_0 ^{1,2} (\Omega)$ is a solution to (\ref{BVP}) if for any $v\in W_0^{1,2} (\Omega)$,
\begin{equation}\label{weak solution}
\int_\Omega a(x,u) Du Dv dx =\int_\Omega f v dx.
\end{equation}
\end{definition}

We note that, because of assumption (\ref{conditon for a(x,s)}), the differential operator
$$
-\mbox {div} (a(x,v)Dv)
$$
is not coercive on $W_0^{1,2} (\Omega)$, even if it is well defined between $W_0^{1,2} (\Omega)$ and its dual, see the last chapter in the monograph \cite{Boccardo-Croce}. We mention that, for $f\in L^m(\Omega)$ and (\ref{conditon for a(x,s)}) is replaced by
$$
\frac {\alpha}{(1+|s|)^\theta } \le a (x,s) \le \beta,
$$
Boccardo, Dall'Aglio
and Orsina derived in \cite{BDO} some existence and regularity results of (\ref{BVP}), see also the last chapter in \cite{Boccardo-Croce}.


We prove the following

\begin{theorem}\label{theorem 1-24}
Let $f\in L_w ^m (\Omega)$, $m>(2^*)'$, and $u\in W_0^{1,2} (\Omega)$ be a solution to (\ref{BVP}). Then

(i) $m>\frac n 2, 0\le \theta <1 \Rightarrow \exists L=L(n,\alpha,\beta, \mu, \|f\| _{L_{w} ^m (\Omega)},m, |\Omega|)>0$ such that $|u| \le 2L$;

(ii) $m =\frac n 2, 0\le \mu \theta <1,  \Rightarrow \mbox {for any } \rho <1-\theta, \exists \lambda =\lambda (n,\alpha, \|f\| _{L_w ^m (\Omega)},m, \theta, \mu) >0 \mbox { such that }  e^{\lambda |u| ^{\rho}} \in L^1(\Omega) $;

(iii) $ (2^*)'< m<\frac n 2, 0\le \theta <1 \Rightarrow \mbox {for any } \varrho <\frac {\alpha (1-\theta)}{1-\beta} =\frac {nm}{n-2m} (1-\theta)=m^{**} (1-\theta), \mbox { one has }  |u| \ln (e+|u|)\in L_w ^{\varrho} (\Omega)$.
\end{theorem}

We now use Corollary \ref{Stampacchia Lemma, generalization-2} to prove Theorem \ref{theorem 1-24}.

\vspace{3mm}

\noindent {\it Proof of Theorem \ref{theorem 1-24}.} Denote $G_k(u) =u-T_k(u)$.
For $h>k >0$, we take
$$
v =T_{h-k} (G_k (u))
$$
as a test function in (\ref{weak solution}) and we have, by (\ref{conditon for a(x,s)}), that
\begin{equation}\label{the first inequality}
\alpha \int_{B_{k,h}} \frac {|Du|^2}{ (1+|u|)^\theta \ln ^\theta (e+|u|)} dx \le \int _{A_k} fT_{h-k} (G_k(u))dx,
\end{equation}
where
$$
B_{k,h} =\{x\in \Omega: k<|u| \le h\}, \ \ A_k =\{x\in \Omega: |u|>k\}.
$$
We now estimate both sides of (\ref{the first inequality}). The left hand side can be estimated as
\begin{equation}\label{the left hand side}
\begin{array}{llll}
&\displaystyle \alpha \int_{B_{k,h}} \frac {|Du|^2}{ (1+|u|)^\theta \ln ^\theta (e+|u|)} dx \\
\ge &\displaystyle \frac {\alpha}{ (1+h) ^\theta \ln ^\theta (e+h) }\int _{B_{k,h}} |Du|^2dx \\
= &\displaystyle \frac {\alpha}{ (1+h) ^\theta\ln ^\theta (e+h) } \int _{\Omega} |DT_{h-k } (G_k(u))|^2dx \\
\ge &\displaystyle \frac {\alpha {\cal S}^2}{(1+h) ^\theta\ln ^\theta (e+h) } \left(\int_\Omega \left|T_{h-k} (G_k(u))\right|^{2^*}dx \right) ^{\frac 2 {2^*}} \\
\ge &\displaystyle \frac {\alpha {\cal S}^2}{(1+h) ^\theta\ln ^\theta (e+h) }\left(\int_{A_h} \left|T_{h-k} (G_k(u))\right|^{2^*}dx \right) ^{\frac 2 {2^*}} \\
= &\displaystyle \frac {\alpha {\cal S}^2}{(1+h) ^\theta\ln ^\theta (e+h) } (h-k) ^2 |A_h| ^{\frac 2 {2^*}},
\end{array}
\end{equation}
where we used again the Sobolev inequality (\ref{Sobolev inequality}).
Using H\"older inequality (\ref{Holder}), the right hand side can be estimated as
\begin{equation}\label{the right hand side}
\int_{A_k}  fT_{h-k} (G_k(u)) dx \le (h-k) \int_{A_k } |f| dx \le (h-k) B  |A_k| ^{\frac 1 {m'}},
\end{equation}
where $B$ is a constant depending on $\|f\|_{L_{w}^m (\Omega)}$ and $m$. (\ref{the first inequality}) together with (\ref{the left hand side}) and (\ref{the right hand side}) implies, for $h>k\ge 1$,
$$
|A_h| \le \frac {c_{4}  (1+h) ^{\frac {\theta 2^*}{2}} \ln  ^{\frac {\theta 2^*}{2}} (e+h)}{ (h-k) ^{\frac {2^*}{2}}}
|A_k| ^{\frac {2^*}{2m'}}\le \frac {c_{5} h^{\frac {\theta 2^*}{2}} \ln ^{\frac {\theta 2^*}{2}} (e+h)}{ (h-k) ^{\frac {2^*}{2}}} |A_k| ^{\frac {2^*}{2m'}},
$$
where $c_{4}, c_{5}$ are constants depending on $n,\alpha, \|f\| _{L_{w} ^m (\Omega)},m$. Thus (\ref{stampacchia_assumption 222}) holds true with
$$
\varphi (k) =|A_k|, \ c=c_5, \ \alpha =\frac {2^*}{2},  \ \beta =\frac {2^*}{2m'} \ \mbox { and } \ k_0=1.
$$
We use Corollary \ref{Stampacchia Lemma, generalization-2} and we have:

(i) If $m>\frac n 2$, then $\beta>1$. In case of $0\le \theta <1$, we use Corollary \ref{Stampacchia Lemma, generalization-2}-(i) and we have $|A_{2L}| =0$ for some constant $L$ depending on $n,\alpha,\beta, \mu, \|f\| _{L_{w} ^m (\Omega)},m$ and $|\Omega|$, from which we derive $|u|\le 2L$ a.e. $\Omega$;

(ii)  If $m=\frac n 2$, then $\beta=1$. For any $\rho <1-\theta$, one can choose $\tilde \theta$ sufficiently close to 1 such that $\rho =1-\frac {\theta}{\tilde \theta}$.
 In case of $\mu \theta <1$,  we use Corollary \ref{Stampacchia Lemma, generalization-2}-(ii) and we derive that there exists a constant $\tau$ depending on $n,\alpha, \|f\| _{L_w ^m (\Omega)},m, \theta$ and $\mu$, such that
$$
|\{|u|>k\}| \le |\{ |u|>1\}| e ^{1-\left(\frac {k-1}{\tau} \right) ^{\rho }} \le |\Omega| e ^{1-\left(\frac {k-1}{\tau} \right) ^{\rho }}.
$$
We let $2^{1+\rho } \lambda =\tau ^{-\rho}$ and $k\ge 2$ ($\Leftrightarrow k-1\ge \frac k 2$) and we have
$$
|\{|u|>k\}| \le |\Omega| e ^{1-2^{1+\rho }\lambda (k-1)^{\rho }} \le |\Omega| e ^{1-2^{1+\rho }\lambda \left(\frac k 2 \right) ^{\rho}}=c_6 e ^{-2\lambda k ^{\rho}}, \ k\ge 2,
$$
where $c_6= |\Omega| e $. Hence
$$
|\{e ^{\lambda |u| ^{\rho}} >e ^{\lambda k ^{\rho}}\}| = |\{ |u| >k\}| \le c _6 e ^{-2\lambda k ^{\rho}}.
$$
Let $ \tilde k =e ^{\lambda k ^{\rho}}$, then
$$
|\{e ^{\lambda |u| ^{\rho}} >  \tilde  k \}| \le \frac {c_6}{  \tilde k ^2}, \ \mbox { for } \  \tilde k \ge e ^{\lambda 2 ^{\rho }}.
$$
Let $\big [e ^{\lambda 2 ^{\rho}}\big] =k_1$.  We now use Lemma 3.11 in \cite{Boccardo-Croce} again for $g= e ^{\lambda |u| ^{\rho}}$ and $r=1$. Since
$$
\sum_{\tilde k=1} ^\infty |\{e ^{\lambda |u| ^{\rho}} >\tilde   k \}|= \left (\sum_{\tilde  k=1} ^ {k_1}+ \sum _{\tilde k=k_1+1} ^\infty\right) |\{e ^{\lambda |u| ^{\rho}} >\tilde   k \}|\le k_1|\Omega| + c _6\sum _{\tilde k=k_1+1} ^\infty \frac 1 {
\tilde  k ^2} <+\infty,
$$
then $e ^{\lambda |u| ^{\rho}} \in L^1 (\Omega)$;

(iii) If $(2^*)'< m<\frac n 2$, then $\beta<1$. For any $\varrho <\frac {\alpha(1-\theta)}{1-\beta} =m^{**} (1-\theta)$, let us take $\epsilon_0$ sufficiently close to $1-\theta$ such that $\varrho =\frac {\alpha \epsilon_0}{1-\beta}$. In case of $0\le \theta <1$ we use Corollary \ref{Stampacchia Lemma, generalization-2}-(iii) and we have for all $k\ge 1$,
$$
|\{|u| >k\}| \le c_7 \left(\frac 1 {g(k)} \right) ^\varrho ,
$$
where $c_7$ is a constant depending on $n,\alpha, \beta, \mu, \|f\| _{L_{w} ^m (\Omega)},m$ and $|\Omega|$.
The above inequality is equivalent to
$$
|\{|g(|u|)| >k\}| \le c_7 \left(\frac 1 k \right) ^\varrho , \ \ \forall k \ge 1.
$$
Thus
$$
|\{|g(|u|)| >k\}| \le \max \{c_7,|\Omega|\} \left(\frac 1 k \right) ^\varrho , \ \ \forall k >0,
$$
that is $g(|u|)\in L_w ^{\varrho } (\Omega)$, as desired.     \qed

\vspace{4mm}

 \noindent {\bf Acknowledgments:} The corresponding author is grateful for the support from the NSFC (12071021)
and the Key Science and Technology Project of Universities of Hebei Province (ZD2021307).

\rm 

\end{document}